\documentclass [10pt,article,oneside]{memoir}

\usepackage[english]{babel}
\usepackage[utf8]{inputenc}
\usepackage{amsmath,amsfonts,mathrsfs,paralist,amssymb,bm,amsthm,graphicx,color,url,mathrsfs,xcolor}

\stockaiv\pageaiv%
\settypeblocksize{22cm}{13cm}{*}%
\setlrmargins{4cm}{*}{1}%
\setmarginnotes{0pt}{0pt}{0pt}%
\setulmargins{3.5cm}{*}{1}%
\setheadfoot{3\baselineskip}{3\baselineskip}%
\setheaderspaces{2\baselineskip}{*}{1}%
\checkandfixthelayout%

\usepackage[all]{xy}
\xyoption{poly}
\usepackage{hyperref}
\hypersetup{colorlinks=false,hidelinks=true}

\newcommand{\addpoint}[1]{#1\ ---\ }

\setsecnumdepth{subsubsection}

\setsecnumformat{\csname the#1\endcsname---}
\setsubsechook{\setsecnumformat{\csname the##1\endcsname\ ---\ }}
\setsechook{\setsecnumformat{\csname the##1\endcsname---}}

\setsecheadstyle{\centering\Large\bfseries\sffamily}
\setsubsecheadstyle{\large\bfseries\sffamily}
\setsubsubsecheadstyle{\itshape\addpoint}
\setsubsubsecindent{1em}
\setbeforesubsubsecskip{0em}
\setsubsubsechook{\setsecnumformat{{\normalfont\csname
      the##1\endcsname\ }}}
\setaftersubsubsecskip{-0em}
\setsubparaheadstyle{\itshape}
\newtheoremstyle{thm}
     {1.5ex plus .3ex minus .1ex}
     {1ex plus .3ex minus .1ex}
     {\itshape}
     {}
     {\sffamily}
     {---}
     {0em}
     {$\bullet$\hbox{\ }#1\hbox{\ }#2}

\theoremstyle{thm}

\newtheorem{definition}{Definition}[section]
\newtheorem{theorem}[definition]{Theorem}

\newtheorem{lemma}[definition]{Lemma}
\newtheorem{proposition}[definition]{Proposition}
\newtheorem{corollary}[definition]{Corollary}

\newtheoremstyle{note}
     {1ex plus .3ex minus .1ex}
     {1ex plus .3ex minus .1ex}
     {}
     {}
     {\itshape}
     {.}
     {1em}
     {}
\theoremstyle{note}

\newlength{\remaining}

\newcommand{\verticale}{\ar@{--}[d]}

\def\calli#1{\expandafter\def\csname
  #1\endcsname{\mathcal{#1}}}	
\def\sets#1{\expandafter\def\csname
  bb#1\endcsname{\mathbb{#1}}}	
\def\rebar#1{\expandafter\def\csname #1bar\endcsname{\overline{\csname
      #1\endcsname}}}		
\def\gothify#1{\expandafter\def\csname
  #1#1#1\endcsname{\mathfrak{#1}}}	

\sets{R}	
\sets{N}	
\sets{Z}	
\calli{I}	
\calli{M}	
\calli{T}	
\calli{R}
\calli{U}
\calli{W}
\gothify{F}	
\gothify{G}	
\rebar{M}	

\renewcommand{\SS}{\mathcal{S}}
\newcommand{\LL}{\mathcal{L}}
\newcommand{\FFFbar}{\overline{\FFF}}
\newcommand{\RR}{\mathscr{R}}
\renewcommand{\SS}{\mathcal{S}}

\newcommand{\C}{\mathscr{C}}

\newcommand{\up}[1]{\,\uparrow #1}
\newcommand{\Up}[1]{\,\Uparrow #1}

\newcommand{\pr}{\mathbb{P}}

\newcommand{\prq}{\mathbb{Q}}
\newcommand{\XS}{X}
\newcommand{\Cstar}{\mathfrak{C}}

\newcommand{\slgb}{\mbox{$\sigma$-al}\-ge\-bra}

\newcommand{\BM}{\partial\M}

\newcommand{\tq}{\;:\;}
\newcommand{\leads}{\Rightarrow}
\newcommand{\lleads}{\Leftrightarrow}
\newcommand{\cro}[2]{\langle\langle #1,#2\rangle\rangle}

\newcommand{\rest}[1]{\bigl|_{#1}}
\newcommand{\height}{\tau}

\newcommand{\un}{\mathbf{1}}

\everymath{\normalfont}

\begin{document}

\mainmatter

\begin{center}
\huge\bfseries\sffamily
Markovian Dynamics of Concurrent Systems
\normalfont
\end{center}

\bigskip
\begin{center}
  Samy Abbes\\
IRIF (UMR CNRS 8243)/Universit\'e Paris Diderot \\
\texttt{samy.abbes@univ-paris-diderot.fr}\\[1em]
\end{center}

\begin{abstract}
Monoid actions of trace monoids over finite sets are powerful models of concurrent systems---for instance they encompass the class of $1$-safe Petri nets.
We characterise Markov measures attached to concurrent systems by finitely many parameters with suitable normalisation conditions. These conditions involve polynomials related to the combinatorics of the monoid and of the monoid action. These parameters generalise to concurrent systems the coefficients of the transition matrix of a Markov chain.

A natural problem is the existence of the uniform measure for every concurrent system. We prove this existence under an irreducibility condition. The uniform measure of a concurrent system is characterised by a real number, the characteristic root of the action, and a function of pairs of states, the Parry cocyle. A new combinatorial inversion formula allows to identify a polynomial of which the characteristic root is the smallest positive root. Examples based on simple combinatorial tilings are studied.
\end{abstract}

\begin{flushleft}
\end{flushleft}

\section{Introduction}
\label{sec:introduction}

Concurrency theory covers the study of systems with concurrency features. An important class of models introduced for this purpose is that of \emph{trace monoids}, also called  \emph{free partially commutative monoids} and \emph{heaps monoids} in the literature~\cite{cartier69,viennot86,diekert90,diekert95,VNBi,krob03}. A trace monoid is a presented monoid $\M$ of the form:
\begin{gather*}
  \M=\langle\Sigma\;|\;ab=ba\rangle\,,
\end{gather*}
for $(a,b)$ ranging over a fixed symmetric and irreflexive relation on the set $\Sigma$ of generators. It allows to express the concurrency of two ``actions'' $a$ and~$b$---whatever their exact nature is---through the commutativity relation $ab=ba$. It renders the essential feature that the order of execution of these two actions is irrelevant. 

In most relevant applications, the notion of \emph{state} of a system
is essential. However, trace monoids are stateless models: given some sequence of actions $(\alpha_1,\dots,\alpha_n)$ corresponding to the element $\alpha_1\cdot\ldots\cdot\alpha_n$ of the monoid~$\M$, any new action $a\in\Sigma$ can be immediately executed, resulting in the new element $\alpha_1\cdot\ldots\cdot\alpha_n\cdot a\in\M$. A natural way to enrich a trace monoid $\M$ with a notion of state is to consider a \emph{monoid action} of $\M$ on the desired set of states---of course, specifying the relations between states and actions is still the responsibility of the model designer.

Hence, the concurrent systems we will consider in this paper have the form of a
pair $(\XS,\M)$, where $X$ is a finite set of states, and $\M$ is a
trace monoid together with a right monoid action
$\varphi:\XS\times\M\to\XS$ of $\M$ over~$\XS$, denoted by
$\varphi(\alpha,x)=\alpha\cdot x$. If $1$ denotes the unit element of
the monoid, the monoid action obeys the two following axioms:
\begin{align*}
  \alpha\cdot 1&=\alpha\,,&\alpha\cdot(x\cdot y)&=(\alpha\cdot x)\cdot y\,.
\end{align*}

Before explaining how the dynamics of such systems can be introduced, it is worth observing that two particular instances of them are well-known in classical systems theory and  in Concurrency theory. The first instance is that of acceptor graphs. Indeed, let $A$ be a finite set, and let $M$ be an $0$-$1$-matrix over~$A$. Hence $M=(M_{i,j})_{(i,j)\in A\times A}$
satisfies $M_{i,j}\in\{0,1\}$, and the system can go\footnote{A note on terminology: we avoid the use of the word \emph{deterministic} since it has different meanings according to the scientific community that uses it. In Probability theory, \emph{deterministic} is opposed to \emph{probabilistic}. On the contrary, in Computer science, \emph{deterministic} is opposed to \emph{non-deterministic}, regardless of the existence of a probabilistic context; and \emph{non-deterministic} refers to the existence of a choice in the evolution of the system for a given action. In the Computer science language, all the systems considered in this paper are \emph{deterministic}.} from state $i$ to
state $j$ in one step if and only if $M_{i,j}=1$\,. Let:
\begin{align*}
  \Sigma&=\bigl\{(i,j)\in A\times A\tq M_{i,j}=1\bigr\}\,,
\end{align*}
thought of as the set of admissible elementary actions, and let
$\Sigma^*$ be the free monoid generated by the admissible actions.

Consider a symbol $\bot\notin A$, and put $A'=A\cup\{\bot\}$. Then the
dynamics of the system corresponds to the unique right action
$A'\times\Sigma^*\to A'$ such that:
\begin{align*}
  \forall i\in A'\quad\forall (j,k)\in \Sigma\quad i\cdot(j,k)=
  \begin{cases}
    \bot,&\text{if $i\neq j$ or if $i=\bot$}\\
k,&\text{if $i=j$}
  \end{cases}
\end{align*}

The acceptor graph defines thus a ``partial action'' of the free
monoid $\Sigma^*$ over the set of states~$A$. The action is
``partial'' in the sense that not all actions are always enabled,
depending on the current state of the system. But, up to adding a
distinguished state~$\bot$, partial actions are actually a particular
instance of a standard action $A'\times\Sigma^*\to A'$, with the
additional feature that the trajectories to be considered are those
that avoid to ever hit the distinguished state~$\bot$. Hence these
correspond to partial actions without concurrency; and at a theoretical
level, partial actions can be treated as normal actions of a free
monoid over a set of states.

The generalisation of ``partial actions'' to concurrent systems will
be central in this work. Our definitive definition of a \emph{concurrent system} will be 
a triple $(X,\M,\bot)$, where $X$ is a finite set of states, $\M$~is a trace
monoid together with a right action $\XS\times\M\to\XS$, and $\bot$ is a forbidden state. ``Trajectories'' to consider will be those that avoid to ever hit the distinguished
state~$\bot$.

Yet another class of concurrent systems has a natural counterpart elsewhere in Concurrency theory. Indeed, it is very natural to encode $1$-safe Petri nets~\cite{reisig85,desel95} as concurrent systems defined above: see~\cite{abbes17} for details and examples.

\medskip
We introduce the dynamics of concurrent systems by means of probability.  We define a
\emph{Markov measure} over a concurrent system~$(\XS,\M,\bot)$, as a
family $\pr=(\pr_\alpha)_{\alpha\in\XS}$ of probability measures
indexed by the set of states, and obeying the following chain rule,
without giving the full definition of all terms for now:
\begin{gather}
  \label{eq:50}
  \forall\alpha\in\XS\quad\forall x,y\in\M\quad\pr_\alpha(x\cdot
  y)=\pr_\alpha(x)\pr_{\alpha\cdot x}(y).
\end{gather}

Furthermore, we shall impose that $\pr_\alpha(x)=0$ holds whenever $\alpha\cdot x=\bot$, \emph{i.e.}, whenever $x$ is a forbidden action in state~$\alpha$.

If $X$ is a singleton, or if $\pr_\alpha$ is independent of~$\alpha$,
then the property becomes $\pr(x\cdot y)=\pr(x)\pr(y)$, which
corresponds to Bernoulli measures---or memoryless measures. Bernoulli
measures on trace monoids have been the topic of a previous work
co-authored with J.~Mairesse~\cite{abbesmair14}. The present work
covers the generalisation to the Markovian case.

The defining property~(\ref{eq:50}) is natural. It corresponds to the intuitive Markov property that the probabilistic evolution of the system, at any stage, is entirely given by its current state. If $\pr_\alpha(x)>0$ in~(\ref{eq:50}), dividing by $\pr_\alpha(x)$ yields the following formulation in terms of conditional probabilities: $\pr_\alpha(x\cdot y\big|x)=\pr_{\alpha\cdot x}(y)$. In words: starting from the state~$\alpha$, and conditionally to the execution of the trajectory~$x$, the evolution of the system is indistinguishable in probability from the execution of the system starting from the state~$\alpha\cdot x$.

Our first contribution is to characterise all Markov measures associated with a trace monoid action $(\M,\XS)$ through a finite family of probabilistic parameters
with certain normalisation conditions. These normalisation conditions
are polynomial in the parameters; they are expressed by means of a
fundamental combinatorial device, the M\"obius transform,
particularised for the framework of trace monoids. This first contribution does not take into account the notion of forbidden state.

Another topic is the analysis of concurrent systems with a forbidden state~$\bot$, hence with the additional constraint $\pr_\alpha(x)=0$ whenever $x$
is an action not enabled at state~$\alpha$. The existence of Markov measures respecting this constraint is not obvious, even with the general description of Markov measures at hand. For acceptor graphs---hence, without the concurrency
feature---, the Parry construction \cite{parry64,kitchens97,lind95}
demonstrates the existence of a ``uniform'' measure, Markovian, and
charging with positive probability every admissible finite trajectory
if the acceptor graph is irreducible. The generalisation of this
measure to concurrent systems is the second contribution of this
paper.

We construct the uniform measure for irreducible concurrent systems and we show that this is indeed a Markov measure $\nu=(\nu_\alpha)_{\alpha\in\XS}$. It has the
following special form, for $x$ enabled at state~$\alpha$:
\begin{align*}
  \nu_\alpha(x)&=t_0^{|x|}\Gamma(\alpha,\alpha\cdot x),\qquad\text{with }
|x|=\text{length of $x$.}
\end{align*}
Here, $t_0$ is a real number lying in $(0,1)$ and is called the
\emph{characteristic root} of the concurrent system; and
$\Gamma(\cdot,\cdot):\XS\times\XS\to\bbR$ is a positive function,
called the \emph{Parry cocyle}, with the following property:
$\Gamma(\alpha,\beta)\Gamma(\beta,\gamma)=\Gamma(\alpha,\gamma)$. 


Determining $t_0$ and the Parry cocycle is challenging when facing
concrete examples. We introduce a simple combinatorial example of
concurrent systems, related to tiling models. We illustrate our
theoretical results by solving the problem of determining both the
characteristic root and the Parry cocycle in several different ways
for this example.

The growth series $\sum_{x\in\M}t^{|x|}$ of a trace monoid $\M$ is a
well studied object. It is given by the inverse of the M\"obius
polynomial of the monoid (also called the independence polynomial in a
graph theoretic context~\cite{levit05}). This directly relates the
radius of convergence of the series with the root of smallest modulus
of the M\"obius polynomial. One method for obtaining this inversion
formula is based on the theory of formal series over partially
commutative variables~\cite{cartier69}. In our context, where not only
the monoid but also its monoid action over a finite set is involved, we show that a slight generalisation of this theory allows to obtain similar results. It is thus another contribution of this paper
to introduce \emph{formal fibred series over non commutative
  variables}, and to show their application for determining the radius
of convergence of growth series of the form:
\begin{gather*}
  Z_\alpha(t)=\sum_{\substack{x\in\M\tq\\\text{$x$ enabled at $\alpha$}}}t^{|x|}\,.
\end{gather*}

The characteristic root $t_0$ corresponds to the common radius of
convergence of the above power series, for $\alpha$ ranging over the
set of states.  We obtain a generalisation of the M\"obius polynomial
of a trace monoid, of which $t_0$ is the smallest root. This amounts
to encoding not only the combinatorics of the monoid, but also of the
action, into a certain polynomial with integer coefficients.

\medskip

\subsubsection*{Organisation}
\label{sec:organisation}
\ In order to keep the paper self-contained, we have gathered in
Section~\ref{sec:preliminaries} some background on trace monoids. We
have also included a summary on Bernoulli measures for trace monoids;
although we shall not use the related results directly, their
comparison with subsequent results on Markov measures will probably be
useful to the reader.

Section~\ref{sec:acti-part-acti} introduces trace monoid actions and concurrent systems.

Section~\ref{sec:markov-measures} is devoted to a general study of
Markov measures for actions of trace monoids.  Probabilistic parameters are characterised and Markov
measures are given a realisation through the Markov chain of
states-and-cliques.

In Section~\ref{sec:uniform-measures}, we consider 
the case of concurrent systems where not every action is necessarily always enabled. The goal is to show the existence of a
Markov measure with the same support as the one of the monoid action. To this aim, we
construct the uniform measure related to the concurrent system, and we show that
it is Markovian. In the course of this construction, we introduce the
characteristic root of the concurrent system and its Parry cocycle. We give an
inversion formula which allows to interpret the characteristic root of
the concurrent system as a particular root of a polynomial attached to the concurrent system.

\section{Preliminaries}
\label{sec:preliminaries}

In this section we collect the needed facts concerning trace monoids
and the construction of associated Bernoulli measures. Bibliographical references are gathered in~\S~\ref{sec:references}.

\subsection{Trace monoids}
\label{sec:trace-monoids}

\subsubsection{Independence relation}
\label{sec:definition}

An alphabet $\Sigma$ is a finite non empty set, and we will always
consider that $|\Sigma|\geq2$\,. Elements of $\Sigma$ are called
letters or pieces. An independence relation $I$ is a binary
irreflexive and symmetric relation on~$\Sigma$.



\subsubsection{Trace monoid}
\label{sec:trace-monoid}

Let $\RR_I$ be the smallest congruence on the free monoid $\Sigma^*$
containing all pairs $(ab,ba)$ for $(a,b)$ ranging over~$I$. The trace
monoid $\M=\M(\Sigma,I)$ is the quotient monoid $\M=\Sigma^*/\RR_I$\,.

\subsubsection{Traces}
\label{sec:traces}

Elements of $\M$ are called traces. Letters of $\Sigma$ are
identified with their images in $\M$ through the canonical morphisms
$\Sigma\to\Sigma^*\to\M$. Concatenation in $\M$ is denoted with the
dot~``$\cdot$''.  The unit element is denoted~``$1$'' and is called
the empty trace.

\subsubsection{Immediate equivalence}
\label{sec:immed-equiv}

Let $\R_I$ be the binary relation on $\Sigma^*$ containing all pairs
of words of the form $(xaby,xbay)$, for $x,y\in\Sigma^*$ and $(a,b)\in
I$. Then $\RR_I$ is the reflexive and transitive closure of~$\R_I$\,. 

\subsubsection{Length of traces}
\label{sec:length-traces}

Since two $\R_I$-related words have the same length, and since $\RR_I$
is the transitive and reflexive closure of~$\R_I$\,, all representative
words of a given trace have the same length. This defines thus a
mapping $|\cdot|:\M\to\bbN$, and $|x|$ is called the length of
trace~$x$. The length is additive: $|x\cdot y|=|x|+|y|$.  The empty
trace is the unique trace of length~$0$, and letters are the only
traces of length~$1$.

\subsubsection{Divisibility relation}
\label{sec:divis-relat}

The left divisibility relation in $\M$ is denoted~``$\leq$'': $x\leq
y\iff\exists z\in\M\quad y=x\cdot z$. As in any monoid, it is
reflexive and transitive. From the additivity of length, it is also
easily seen to be anti-symmetric, hence $(\M,\leq)$ is a partially
ordered set.

\subsubsection{Compatible traces}
\label{sec:compatible-traces}

Two traces $x,y\in\M$ have a least upper bound $x\vee y$ with respect
to~$\leq$ if and only if there exists a trace $z\in\M$ such that
$x\leq z$ and $y\leq z$. In this case, $x$~and $y$ are said to be
compatible.

\subsubsection{Cliques}
\label{sec:cliques}

A clique of $\M$ is a trace of the form $x=a_1\cdot\ldots\cdot a_i$\,,
where $a_i$ are pairwise distinct letters such that $i\neq j\implies
(a_i,a_j)\in I$. Consider the pair $(\Sigma,I)$ as a graph. Then
cliques are in bijection with those subgraphs of $(\Sigma,I)$ which
are complete graphs---these are indeed called cliques in a graph
theoretic context. Restricted to cliques, the ordering relation $\leq$
corresponds to the inclusion of subsets, when seeing cliques as
subsets of~$\Sigma$.  We denote by $\C$ the set of cliques, and by
$\Cstar=\C\setminus\{1\}$ the set of non empty cliques.

\subsubsection{Parallelism of cliques}
\label{sec:parallelism-cliques}

Two cliques $c,c'\in\C$ are said to be parallel, denoted $c\parallel
c'$, whenever $c\cdot c'\in\C$. When seeing cliques as subsets
of~$\Sigma$, this is equivalent to: $c\times c'\subseteq I$. In
particular, and since single letters are cliques by themselves, we
have $a\parallel a'\iff (a,a')\in I$ if $a,a'\in\Sigma$.

\subsubsection{Cartier-Foata relation}
\label{sec:cart-foata-relat}

Two cliques $c,c'\in\C$ are Cartier-Foata compatible, denoted by $c\to
c'$, if for every letter $b\in c'$ there exists a letter $a\in c$ such
that $(a,b)\notin I$. 

\subsubsection{Cartier-Foata normal form. Height of traces}
\label{sec:cartier-foata-normal}

For every non empty trace $x\in\M\setminus\{1\}$, there exists a
unique integer $n\geq1$ and a unique sequence $(c_1,\ldots,c_n)$ of
non empty cliques such that:
\begin{inparaenum}[(1)]
  \item $c_i\to c_{i+1}$ holds for all $i\in\{1,\ldots,n-1\}$; and
  \item $x=c_1\cdot\ldots\cdot c_n$\,. 
\end{inparaenum}
The sequence $(c_1,\ldots,c_n)$ is called the Cartier-Foata decomposition
(or normal form) of~$x$.

The integer $n$ is called the height of~$x$, denoted by
$n=\height(x)$. By convention, we put $\height(1)=0$. 

\subsubsection{Heaps of pieces}
\label{sec:heaps-pieces}

Heaps of pieces provide a visually intuitive representation of
traces. Picture each letter as a piece, or domino, falling from top to
bottom until it reaches the ground or a previously placed piece. Two
pieces $a,b\in\Sigma$ are bound to fall in a parallel way with respect
to each other if and only if $(a,b)\in I$, which renders the
commutativity relation $a\cdot b=b\cdot a$. Then it is part of
Viennot's theory that heaps of pieces thus obtained are in bijection
with traces. We illustrate the heap representation in
Figure~\ref{fig:pojqpqaaa}.

\begin{figure}
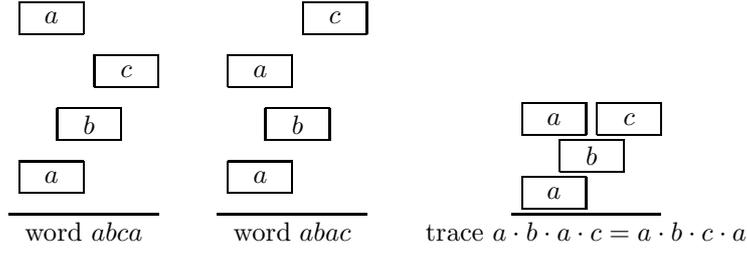

  \centering
  \begin{tabular}{ccc}
\xy
<.1em,0em>:
(0,6)="G",
"G"+(12,6)*{a},
"G";"G"+(24,0)**@{-};"G"+(24,12)**@{-};"G"+(0,12)**@{-};"G"**@{-},
(14,26)="G",
"G"+(12,6)*{b},
"G";"G"+(24,0)**@{-};"G"+(24,12)**@{-};"G"+(0,12)**@{-};"G"**@{-},
(0,66)="G",
"G"+(12,6)*{a},
"G";"G"+(24,0)**@{-};"G"+(24,12)**@{-};"G"+(0,12)**@{-};"G"**@{-},
(28,46)="G",
"G"+(12,6)*{c},
"G";"G"+(24,0)**@{-};"G"+(24,12)**@{-};"G"+(0,12)**@{-};"G"**@{-},
(-4,-2);(52,-2)**@{-}
\endxy& \quad
\xy
<.1em,0em>:
(0,6)="G",
"G"+(12,6)*{a},
"G";"G"+(24,0)**@{-};"G"+(24,12)**@{-};"G"+(0,12)**@{-};"G"**@{-},
(14,26)="G",
"G"+(12,6)*{b},
"G";"G"+(24,0)**@{-};"G"+(24,12)**@{-};"G"+(0,12)**@{-};"G"**@{-},
(0,46)="G",
"G"+(12,6)*{a},
"G";"G"+(24,0)**@{-};"G"+(24,12)**@{-};"G"+(0,12)**@{-};"G"**@{-},
(28,66)="G",
"G"+(12,6)*{c},
"G";"G"+(24,0)**@{-};"G"+(24,12)**@{-};"G"+(0,12)**@{-};"G"**@{-},
(-4,-2);(52,-2)**@{-}
\endxy
&\quad \xy
<.1em,0em>:
0="G",
"G"+(12,6)*{a},
"G";"G"+(24,0)**@{-};"G"+(24,12)**@{-};"G"+(0,12)**@{-};"G"**@{-},
(14,14)="G",
"G"+(12,6)*{b},
"G";"G"+(24,0)**@{-};"G"+(24,12)**@{-};"G"+(0,12)**@{-};"G"**@{-},
(0,28)="G",
"G"+(12,6)*{a},
"G";"G"+(24,0)**@{-};"G"+(24,12)**@{-};"G"+(0,12)**@{-};"G"**@{-},
(28,28)="G",
"G"+(12,6)*{c},
"G";"G"+(24,0)**@{-};"G"+(24,12)**@{-};"G"+(0,12)**@{-};"G"**@{-},
(-4,-2);(52,-2)**@{-}
\endxy\\
word $abca$&\quad word $abac$&\quad trace $a\cdot b\cdot a\cdot c=a\cdot b\cdot c\cdot a$
\end{tabular}
\caption{\small\textsl{Two congruent words
    and the resulting heap (trace) for~$\M=\langle a,b,c\ |\ a\cdot
    c=c\cdot a\rangle$\,. Pieces $a$ and $c$ fall in a parallel way,
    but not $a$ and~$b$ and neither $c$ and~$b$.}}
\label{fig:pojqpqaaa}
\end{figure}

\subsubsection{Layers of heaps and Cartier-Foata normal form}
\label{sec:layers-heaps-cartier}

In the heap of pieces representation of traces, the cliques that appear
in the Cartier-Foata decomposition correspond to the successive
horizontal layers, from bottom to top, that compose the heap. For the
trace depicted in Figure~\ref{fig:pojqpqaaa}, the Cartier-Foata
decomposition is $a\to b\to a\cdot c$.

\subsection{Growth series and M\"obius inversion formulas}
\label{sec:growth-series-mobius}

\subsubsection{Growth series and M\"obius polynomial}
\label{sec:growth-series}
\label{sec:mobius-polynomial}

The growth series of the monoid $\M$ is the power series $H(t)$ defined by:
\begin{gather*}
  H(t)=\sum_{x\in\M}t^{|x|}\,.
\end{gather*}



The M\"obius polynomial $\mu_\M(t)$ of $\M$ is defined by:
\begin{gather*}
  \mu_\M(t)=\sum_{\gamma\in\C}(-1)^{|\gamma|}t^{|\gamma|}\,.
\end{gather*}
For instance, for $\M=\langle a,b,c\ |\ a\cdot c=c\cdot a\rangle$,
illustrated in Figure~\ref{fig:pojqpqaaa}, one has
$\C=\{1,a,b,c,a\cdot c\}$ and thus $\mu_\M(t)=1-3t+t^2$\,.

\subsubsection{First M\"obius inversion formula}
\label{sec:first-mobi-invers}

As a formal series, the growth series $H(t)$ is rational, inverse of
the M\"obius polynomial:
\begin{gather}
\label{eq:2}
  H(t)=\frac1{\mu_\M(t)}\,.
\end{gather}

\subsubsection{Smallest root of the M\"obius polynomial}
\label{sec:domin-sing-zt}

The M\"obius polynomial has a unique root of smallest modulus, which
is real and lies in~$(0,1)$. This root coincides with the unique
dominant singularity of~$H(t)$. 



\subsubsection{M\"obius transform}
\label{sec:mobius-transform}

Let $f:\C\to\bbR$ be a function (it could actually take its values in
any commutative monoid). The M\"obius transform of $f$ is the function
$h:\C\to\bbR$ defined by:
\begin{gather}
\label{eq:3}
\forall c\in\C\quad h(c)=\sum_{c'\in\C\tq c'\geq
  c}(-1)^{|c'|-|c|}f(c')\,.
\end{gather}

\subsubsection{Graded M\"obius transform}
\label{sec:grad-mobi-transf}

If $f:\M\to\bbR$ is defined on~$\M$, and not only on~$\C$, we extend
its M\"obius transform as follows. It is defined as in~(\ref{eq:3})
on~$\C$. For $\height(x)\geq2$, let $x=c_1\to\ldots\to c_n$ be the
Cartier-Foata decomposition of~$x$, and let $y=c_1\cdot\ldots\cdot
c_{n-1}$\,. Then we define:
\begin{gather*}
  h(x)=\sum_{c\in\C\tq c\geq c_n}(-1)^{|c|-|c_n|}f(y\cdot c)\,.
\end{gather*}
The function $h:\M\to\bbR$ is called the graded M\"obius transform
of~$f$. 

\subsubsection{Second M\"obius inversion formula}
\label{sec:second-mobi-invers}

Define $\M(1)=\C$ and, for each non empty trace $x\in\M$, put
$\M(x)=\{y\in\M\tq\height(y)=\height(x)\wedge x\leq y\}$\,.  Then, for
any function $f:\M\to\bbR$, with graded M\"obius transform
$h:\M\to\bbR$, holds:
\begin{gather}
\label{eq:17}
\forall x\in\M\quad f(x)=\sum_{y\in\M(x)}h(y)\,.
\end{gather}

Conversely, if $f,h:\M\to\bbR$ are two functions such
that~(\ref{eq:17}) holds, then $h$ is the graded M\"obius transform
of~$f$.

\subsection{Boundary and compactification}
\label{sec:compactification}

\subsubsection{Generalised traces}
\label{sec:compactification-1}

There exists a canonical partial order \mbox{$(\Mbar,\leq)$}, which
elements are called generalised traces, and with the following
properties:
\begin{enumerate}
\tightlist
\item\label{item:1} Every non decreasing sequence in $\Mbar$ has a
  least upper bound.
\item\label{item:2} There is a canonical embedding of partial orders
  $\iota:\M\to\Mbar$, hence we identify $\M$ as a subset of~$\Mbar$.
\item Every element of $\Mbar$ is the least upper bound
  $\bigvee\{x_n\tq n\geq1\}$ of a non decreasing sequence
  $(x_n)_{n\geq1}$ with $x_n\in\M$ for all $n\geq1$.
\end{enumerate}

\subsubsection{Infinite traces}
\label{sec:infinite-traces}

Let $\BM=\Mbar\setminus\M$. Elements of $\BM$ are called infinite
traces, and $\BM$ is the boundary at infinity, or simply the boundary
of~$\M$. For every $\xi\in\BM$, there exists a unique infinite
sequence $(c_i)_{i\geq1}$ of non empty cliques such that holds:
\begin{align*}
  \xi&=\bigvee_{i\geq1}(c_1\cdot\ldots\cdot c_i)\,,&\forall i\geq
  1\quad c_i&\to c_{i+1}\,.
\end{align*}
The sequence $(c_i)_{i\geq1}$ extends to infinite traces the
Cartier-Foata decomposition of traces (\S~\ref{sec:cartier-foata-normal}). 



\subsubsection{Ordering on generalised traces}
\label{sec:order-gener-trac-1}

For each integer $n\geq0$, the $n$-topping is the mapping
$\kappa_n:\Mbar\to\M$ defined by $\kappa_n(\xi)=c_1\cdot\ldots\cdot
c_n$\,, where $(c_i)_{i\geq1}$ is the extended Cartier-Foata
decomposition of~$\xi$, maybe with $c_i=1$ whenever $\xi\in\M$ and
$i>\height(\xi)$.

The ordering on traces is defined by the left divisibility relation
(\S~\ref{sec:divis-relat}). Its extension on $\Mbar$ can be
characterised as follows: for all $\xi,\xi'\in\Mbar$, $\xi\leq\xi'$
holds if and only if $\kappa_n(\xi)\leq\kappa_n(\xi')$ for all
integers $n\geq0$.  For $x\in\M$ and $\xi\in\BM$, the relation
$x\leq\xi$ reduces to this:
\begin{gather*}
  x\leq\xi\iff x\leq\kappa_{\height(x)}(\xi)\,.
\end{gather*}
The visual intuition of this result is illustrated on
Figure~\ref{fig:pppqqqaax}.

\newsavebox{\mycaptioni}\savebox{\mycaptioni}{\raisebox{.5ex}{\xy<.1em,0em>:(0,0);(10,0)**@{.}\endxy}}
\newsavebox{\mycaptionii}\savebox{\mycaptionii}{\raisebox{.5ex}{\xy<.1em,0em>:(0,0);(10,0)**@{-}\endxy}}

\begin{figure}
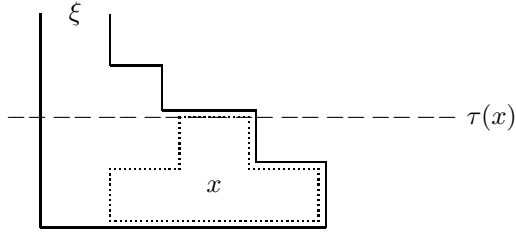

  \centering
$\xy
<.13em,0em>:
(20,0);
(20,15)**@{.};(40,15)**@{.};(40,30)**@{.};(60,30)**@{.};(60,15)**@{.};(80,15)**@{.};(80,0)**@{.};(20,0)**@{.};
(0,-2);(0,60)**@{-};
(20,60);(20,45)**@{-};(35,45)**@{-};(35,32)**@{-};(62,32)**@{-};(62,17)**@{-};(82,17)**@{-};(82,-2)**@{-};(0,-2)**@{-};
(-10,30);(120,30)**{--},(130,30)*{\height(x)};
(10,60)*{\xi},(50,10)*{x}
\endxy$
\caption{\small\textsl{Illustrating the ordering relation $x\leq\xi$
    between a trace~$x$, symbolised by the dotted
    line~\usebox{\mycaptioni}\,, and an infinite trace~$\xi$,
    symbolised by the solid line~\usebox{\mycaptionii}\,. The height
    $\height(x)$ of $x$ is depicted.}%
}
\label{fig:pppqqqaax}
\end{figure}

\subsubsection{Elementary cylinders}
\label{sec:elem-cylind}

 For $x\in\M$, the elementary cylinder of base $x$ is the subset of
 $\BM$ denoted by $\up x$ and defined by:
 \begin{gather*}
   \up x=\{\xi\in\BM\tq x\leq\xi\}\,.
 \end{gather*}

 The existence of least upper bounds for compatible traces
 (\S~\ref{sec:compatible-traces}) implies that elementary cylinders
 intersect as follows:
\begin{gather*}
  \up x\;\cap \up y=
  \begin{cases}
    \emptyset,&\text{if $x$ and $y$ are not compatible,}\\
\up(x\vee y),&\text{if $x$ and $y$ are compatible.}
  \end{cases}
\end{gather*}

\subsubsection{Topology on $\BM$ and on $\Mbar$. Compactness}
\label{sec:topology-m-mbar}
\label{sec:compactness}

For each $x\in\M$, let the full elementary cylinder $\Up x$ be defined
by:
\begin{gather}
\label{eq:45}
  \Up x=\{\xi\in\Mbar\tq x\leq \xi\}\,.
\end{gather}

The topology we consider on $\Mbar$ corresponds to the Lawson topology
in Domain theory (see~\cite{gierz03}). It is the join of the
topologies $\SS$ and~$\LL$, i.e., the smallest topology containing
both $\SS$ and~$\LL$. The topologies $\SS$ (Scott topology) and $\LL$
(lower topology) are the topologies generated by the following subsets:
\begin{gather*}
\begin{array}{ll}
  \SS:&\text{$\Up x$, for $x\in\M$}\\
  \LL:&\text{$\Mbar\setminus(\Up x)$, for $x\in\Mbar$}
\end{array}
\end{gather*}

The boundary $\BM$ is equipped with the restriction of this topology.
For our concern, we will only need the following facts\ regarding
these topologies:
\begin{enumerate}
\tightlist
\item For each trace $x\in\M$, the singleton $\{x\}$ is both open and
  closed in~$\Mbar$.
\item For each trace $x\in\M$, $\up x$ is both open and closed
  in~$\BM$, and $\Up x$ is both open and closed in~$\Mbar$.
\item The space $\Mbar$ is metrisable and compact; the subset $\BM$ is
  closed in~$\Mbar$.
\end{enumerate}

\subsection{Bernoulli and uniform measures on the boundary}
\label{sec:bern-meas-unif}

\subsubsection{\slgb s and $\pi$-systems on $\Mbar$ and on $\BM$}
\label{sec:slgb}

We equip $\Mbar$ and $\BM$ with their respective Borel \slgb s,
$\FFFbar$~and~$\FFF$. The \slgb\ on $\Mbar$ is generated by the
collection of full elementary cylinders~$\Up x$, for $x$ ranging
over~$\M$, defined in~\S~\ref{sec:topology-m-mbar}.

By the intersection property of cylinders (\S~\ref{sec:elem-cylind},
see also~\S~\ref{sec:compatible-traces}), we observe that both
collections:
\begin{align*}
  \{\emptyset\}&\cup\{\up x\tq x\in\M\}\,,&
\{\emptyset\}&\cup\{\Up x\tq x\in\M\}\,,
\end{align*}
are $\pi$-systems (i.e.: stable under finite intersections) generating
$\FFF$ and~$\FFFbar$, respectively.

\subsubsection{Valuations}
\label{sec:valuations}

A valuation on $\M$ is a function $f:\M\to\bbR$ such that $f(1)=1$ and
$f(x\cdot y)=f(x)f(y)$ for all $x,y\in\M$. Valuations are in bijection
with families or real numbers~$(\lambda_a)_{a\in\Sigma}$\,. The
correspondence assigns to each valuation $f$ the
collection~$(f(a))_{a\in\Sigma}$\,. The valuation is uniform if it is
constant on~$\Sigma$. It corresponds to:
\begin{gather*}
\forall x\in\M\quad  f(x)=p^{|x|}\,,\qquad\forall a\in\Sigma\quad p=f(a)\,.
\end{gather*}

\subsubsection{Bernoulli measures}
\label{sec:bernoulli-measures}

A Bernoulli measure is a probability measure on $(\BM,\FFF)$ such that:
\begin{gather*}
  \forall x,y\in\M\quad\pr(\up (x\cdot y))=\pr(\up x)\pr(\up y)\,.
\end{gather*}
If $\pr$ is a Bernoulli measure, the function $f:x\in\M\mapsto\pr(\up
x)$ is the valuation associated with~$\pr$. 

\subsubsection{M\"obius valuations}
\label{sec:mobius-valuations}

Let $f:\M\to\bbR$ be a valuation, and let $h:\M\to\bbR$ be the
M\"obius transform of~$f$ (\S~\ref{sec:mobius-transform}). By
definition, $f$~is a M\"obius valuation if:
\begin{align*}
  h(1)&=0\,,&\forall c\in\Cstar\quad h(c)\geq0\,.
\end{align*}

\subsubsection{Characterisation of Bernoulli measures}
\label{sec:caract-bern-meas}

If $\pr$ is a Bernoulli measure, then its associated valuation is
M\"obius. Conversely, let $f:\M\to\bbR$ be a M\"obius valuation. Then
there exists a unique Bernoulli measure $\pr$ such that $\pr(\up
x)=f(x)$ for all $x\in\M$.

\subsubsection{Random decomposition of infinite traces}
\label{sec:rand-decomp-infin}

For each integer $i\geq1$, let $C_i:\BM\to\Cstar$ be the mapping which
assigns to an infinite trace $\xi$ the $i^{\text{th}}$~clique in its
extended Cartier-Foata decomposition. Let also:
\begin{gather*}
  Y_i=C_1\cdot\ldots\cdot C_i\,,
\end{gather*}
be defined for all non negative integers, with $Y_0=1$ for $i=0$.
Then $(C_i)_{i\geq1}$ and $(Y_i)_{i\geq0}$ are two sequences of
measurable mappings.

\subsubsection{Markov chain of cliques}
\label{sec:markov-chain-cliques}

Under a Bernoulli measure~$\pr$, the sequence $(C_i)_{i\geq1}$ is a
Markov chain. The law of $C_1$ is given by the
restriction~$h\rest{\Cstar}$\,, where $h:\C\to\bbR$ is the M\"obius
transform of the associated valuation. In case the valuation
$f(\cdot)=\pr(\up\cdot)$ is positive on~$\M$, then the transition
matrix $P=(P_{c,c'})_{(c,c')\in\Cstar\times\Cstar}$ of the Markov
chain $(C_i)_{i\geq1}$ is given by:
 \begin{align*}
   P_{c,c'}&=\un_{\{c\to
     c'\}}\frac{h(c')}{g(c)}\,,&g(c)&=\sum_{c'\in\Cstar\tq c\to c'}h(c')\,.
 \end{align*}

\subsubsection{Uniform measure}
\label{sec:uniform-measure}

There exists a unique Bernoulli measure $\nu$ such that the associated
valuation is uniform. It is given by:
\begin{align*}
  \forall x\in\M\quad\nu(\up x)=p_0^{|x|}\,,
\end{align*}
where $p_0$ is the unique root of smallest modulus of the M\"obius
polynomial~$\mu_\M$ (\S~\ref{sec:mobius-polynomial}).

\subsubsection{Complement on elementary cylinders}
\label{sec:compl-elem-cylind}

It is useful to relate elementary cylinders
(\S~\ref{sec:elem-cylind}) and $Y_i$-measurable subsets of~$\BM$,
where $Y_i$ has been defined in~\S~\ref{sec:rand-decomp-infin}. Let
$x$ be a non empty trace of height
$n=\height(x)$. 
Then the elementary cylinder $\up x$ decomposes as the following
disjoint union:
\begin{align}
  \label{eq:16} \up x&=\bigcup_{z\in\M(x)}\{Y_n=z\}\,,
\end{align}
where $\M(x)=\{z\in\M\tq\height(z)=\height(x)\wedge x\leq z\}$ has
been introduced in~\S~\ref{sec:second-mobi-invers}.

\subsection{Bibliographical references}
\label{sec:references}

The basics of trace monoids are covered, under different points of
view, in \cite{cartier69,viennot86,diekert90,diekert95}. A famous
application of trace monoids to the combinatorics of directed animals
is found in~\cite{bousquet-melou92}.

The first M\"obius inversion formula is established in
\cite{cartier69} and in~\cite{viennot86}. The properties of the
M\"obius polynomial of a trace monoid are studied in
\cite{goldwurm00,krob03,csikvari13}. The second inversion formula for
graded M\"obius transform is established in~\cite{abbesmair14} and in~\cite{abbes17:_unifor} in the context of braid monoids. The
simple M\"obius transform is a particular instance of the notion of
M\"obius transform~\cite{rota64,stanley97}.

The space of generalised traces is introduced
in~\cite{abbesmair14}. It is a particular instance of the completion
of certain presented monoids studied in~\cite{abbes08}, and is also a
particular instance of the completion of partial
orders~\cite{keimel09}. The topological properties are rather
standard; they can, for instance, be established within the framework
of~\cite{abbes06:_bifin}.

Bernoulli and uniform measures for trace monoids are introduced
in~\cite{abbesmair14}, in which only irreducible trace monoids were
considered. Relaxing this assumption does not present major
difficulties, as shown in~\cite{abbes:_unifor}.

\section{Concurrent systems}
\label{sec:acti-part-acti}

\subsection{Actions of trace monoids}
\label{sec:total-actions}

Recall that a \emph{right monoid action}, or simply a \emph{monoid action}, of a monoid $\M$ with unit $1$ over a set $X$ is
defined by a mapping $\varphi:X\times \M\to X$, denoted by
$\varphi(\alpha,x)=\alpha\cdot x$, and satisfying the following
properties:
\begin{inparaenum}[(1)]
  \item $\forall\alpha\in X\quad \alpha\cdot1=\alpha$\,;
  \item $\forall\alpha\in X\quad\forall x,y\in\M\quad (\alpha\cdot
    x)\cdot y=\alpha\cdot(x\cdot y)$\,.
\end{inparaenum}

In the remaining of the paper, we will only consider actions of trace monoids over finite sets, generically denoted by the pair $(\M,\XS)$, where $\M=\M(\Sigma,I)$.  In
  this context, elements of $\XS$ are called \emph{states}, elements of\/
  $\Sigma$ are called \emph{elementary actions}, and elements of $\M$ are
  called \emph{actions}. Intuitively, the state $\alpha\cdot x$ corresponds to the state where the system arrives after having performed the action~$x$, starting from state~$\alpha$.

Assume given an action of a trace monoid $\M=\M(\Sigma,I)$
over~$\XS$. Let $\alpha\in\XS$, and let $a,a'\in\Sigma$ with
$a\parallel a'$ (see~\S~\ref{sec:parallelism-cliques}). Then $a\cdot
a'=a'\cdot a$ and therefore: $(\alpha\cdot a)\cdot a' =(\alpha\cdot
a')\cdot a$. The following proposition states a converse to this
observation.

\begin{proposition}
  \label{prop:1}
Let $\M=\M(\Sigma,I)$ be a trace monoid, let $\XS$ be a finite set, and let
$(\varphi_a)_{a\in\Sigma}$ be a family of mappings
$\varphi_a:\XS\to\XS$ indexed by~$\Sigma$. Assume that holds:
\begin{gather*}
  \forall a,a'\in\Sigma\quad
a\parallel a'\implies\varphi_a\circ\varphi_{a'}=\varphi_{a'}\circ\varphi_a\,.
\end{gather*}
Then there exists a unique action $\varphi:\XS\times\M\to\XS$ such
that $\varphi(\cdot,a)=\varphi_a$ for all $a\in\Sigma$.
\end{proposition}

\begin{proof}
  Immediate by the universal property of quotient monoids.
\end{proof}

\subsection{Concurrent systems and specification through partial actions}
\label{sec:partial-actions}

\subsubsection{Definition}
\label{sec:definition-1}

For several applications, it is desirable that some actions may not be
always ``enabled'', depending on the current state. It is this very
feature that we aim at in introducing the following definition.

\begin{definition}
  \label{def:14}
  A \emph{concurrent system} is a triple $(\M,\XS,\bot)$, where $\M=\M(\Sigma,I)$ is a trace monoid acting on the finite set~$X$, and where $\bot$ is a distinguished element of\/~$X$, satisfying:
  \begin{enumerate}
  \item $\bot\cdot a=\bot$ for all $a\in\Sigma$, and thus $\bot\cdot x=\bot$ for all $x\in\M$; and
  \item For every state $\alpha\in\XS\setminus\{\bot\}$, there exists at least one letter $a\in\Sigma$ such that $\alpha\cdot a\neq\bot$.
  \end{enumerate}
\end{definition}

Here the distinguished state $\bot$ is not thought of as a ``real state'' of the system, but rather as an artefact rendering the fact that it is impossible for the system to execute actions leading to~$\bot$.

Partial actions introduced in the following definition allow for the effective specification of concurrent systems: see Proposition~\ref{prop:2} below.

\begin{definition}
  \label{def:2}
A \emph{partial action} of a trace monoid $\M=\M(\Sigma,I)$ over a finite set
$\XS$ is defined by:
\begin{enumerate}
\item\label{item:3} A collection of subsets
  $(\Sigma(\alpha))_{\alpha\in\XS}$ with
  $\Sigma(\alpha)\subseteq\Sigma$ and $\Sigma(\alpha)\neq\emptyset$,
  where $\Sigma(\alpha)$ corresponds to the set of enabled elementary
  actions at state~$\alpha$.

\item\label{item:4} A collection $(\psi_\alpha)_{\alpha\in\XS}$ of
  mappings, with $\psi_\alpha:\Sigma(\alpha)\to\XS$ denoted
  $\psi_\alpha(a)=\alpha\cdot a$ whenever $a\in\Sigma(\alpha)$, such
  that for any $a,b\in\Sigma$, if $a\in\Sigma(\alpha)$ and if
  $a\parallel b$, then:
  \begin{enumerate}[(i)]
  \item\label{item:5} Either $b\in\Sigma(\alpha)$; and in this case
    $a\in\Sigma(\alpha\cdot b)$ and $b\in\Sigma(\alpha\cdot a)$ and
    $(\alpha\cdot a)\cdot b=(\alpha\cdot b)\cdot a$;
  \item\label{item:6} Or $b\notin\Sigma(\alpha)$; and in this case
    $b\notin\Sigma(\alpha\cdot a)$.
  \end{enumerate}
\end{enumerate}
\end{definition}

The intuition explaining the two properties of point~\ref{item:4} in
the above definition is that an elementary action $b\in\Sigma$ which is not
enabled at some state $\alpha\in\XS$ suffers from some sort of lock on
it. The idea is that a parallel action $a\parallel b$ cannot
unlock~$b$; and neither can it lock $b$ if $b$ was enabled.
See a concrete example in~\S~\ref{sec:tip-top-action} below.

Partial actions can be used for the specifications of concurrent systems, as shown by the following result. See an example of application below in~\S~\ref{sec:tip-top-action}.

\begin{proposition}
  \label{prop:2}
Assume given a partial action with the
same notations as in Definition\/~{\normalfont\ref{def:2}}. Let $\bot$ be an symbol
not in~$\XS$, and let $\XS'=\XS\cup\{\bot\}$. For each $\alpha\in\XS$,
extend $\psi_\alpha:\Sigma(\alpha)\to\XS$ to a mapping
$\psi_\alpha:\Sigma\to\XS'$ by putting $\psi_\alpha(a)=\bot$ if
$a\notin\Sigma(\alpha)$, and put $\psi_\bot(\cdot)=\bot$.

Then there exists a unique monoid action of $\M$ over $\XS'$ such that
  $\alpha\cdot a=\psi_\alpha(a)$ for all $\alpha\in\XS'$ and for all
  $a\in\Sigma$. This monoid action defines a concurrent system $(\M,\XS',\bot)$, said to be associated to the partial action.
\end{proposition}

\begin{proof}
  For each $a\in\Sigma$, let $\varphi_a:\XS'\to\XS'$ be defined by
  $\varphi_a(\alpha)=\psi_\alpha(a)$. We verify that
  $(\varphi_a)_{a\in\Sigma}$ satisfies the condition of
  Proposition~\ref{prop:1}. For this, let $a,b\in\Sigma$ with
  $a\parallel b$, and let $\alpha\in\XS'$. Clearly, for $\alpha=\bot$,
  one has
  $\varphi_a\circ\varphi_b(\bot)=\bot=\varphi_b\circ\varphi_a(\bot)$. And
  for $\alpha\in\XS$, we discuss different cases.

  Assume that neither $a\in\Sigma(\alpha)$ nor $b\in\Sigma(\alpha)$
  hold. Then: $\varphi_a(\varphi_b(\alpha))=\varphi_a(\bot)=\bot$ and
  $\varphi_b(\varphi_a(\alpha))=\varphi_b(\bot)=\bot$, hence
  $\varphi_a\circ\varphi_b(\alpha)=\varphi_b\circ\varphi_a(\alpha)$.

  Assume that both $a\in\Sigma(\alpha)$ and $b\in\Sigma(\alpha)$
  hold. Then, by point~(\ref{item:5}) of Definition~\ref{def:2}, the
  equality
  $\varphi_a\circ\varphi_b(\alpha)=\varphi_b\circ\varphi_a(\alpha)$
  holds.

  Assume that, of $a$ and~$b$, only one is enabled at state~$\alpha$,
  say $a\in\Sigma(\alpha)$ and $b\notin\Sigma(\alpha)$. Then on the
  one hand: $\varphi_a(\varphi_b(\alpha))=\varphi_a(\bot)=\bot$. And
  on the other hand, $b\notin\Sigma(\alpha\cdot a)$ by
  point~(\ref{item:6}) of Definition~\ref{def:2}, and thus:
  $\varphi_b(\varphi_a(\alpha))=\bot$, whence the equality
  $\varphi_a\circ\varphi_b(\alpha)=\varphi_b\circ\varphi_a(\alpha)$.

  The result follows then from Proposition~\ref{prop:1}.
\end{proof}

\subsubsection{Trace language associated with a partial action}
\label{sec:trace-lang-assoc}

We introduce below a collection of trace languages (a trace language
is a subset of a trace monoid) associated with a concurrent system. Proposition~\ref{prop:3} will motivate the introduction of valuations in Section~\ref{sec:mark-meas-fibr}.

\begin{definition}
  \label{def:3}
  Let $(\M,\XS,\bot)$ be a concurrent system. For each state
  $\alpha\in\XS$, we define the subset $\M_\alpha\subseteq\M$ and the
  function $F_\alpha:\M\to\{0,1\}$ by:
\begin{align*}
  \M_\alpha&=\{x\in\M\tq \alpha\cdot x\neq\bot\}\,,&
F_\alpha(x)&=\un_{\{\alpha\cdot x\neq\bot\}}\,.
\end{align*}
The family $F=(F_\alpha)_{\alpha\in\XS}$ is called the support
valuation of the partial action.
\end{definition}

\begin{proposition}
  \label{prop:3}
For each state $\alpha\in\XS$ of a concurrent system $(\M,\XS,\bot)$,
the language $\M_\alpha$ is downward closed:
\begin{gather*}
  \forall x,y\in\M\quad x\in\M_\alpha\wedge
  y\leq x\implies y\in\M_\alpha\,.
\end{gather*}
The support valuation satisfies the following properties:
\begin{gather*}
  \forall\alpha\in\XS\setminus\{\bot\}\quad F_\alpha(1)=1\\
\forall\alpha\in\XS\quad\forall x,y\in\M\quad F_\alpha( x\cdot
y)=F_\alpha(x)F_{\alpha\cdot x}(y)\,.
\end{gather*}

\end{proposition}

\begin{proof}
  By contraposition, assume that $x,y\in\M$ are such that
  $y\notin\M_\alpha$ and $y\leq x$, we have to prove that
  $x\notin\M_\alpha$\,. Let $z\in\M$ be such that $x=y\cdot z$. Then,
  by the action property, one has: $\alpha\cdot x=(\alpha\cdot y)\cdot
  z=\bot\cdot z$ and an easy induction shows that $\bot\cdot z=\bot$
  for all $z\in\M$. Hence $\alpha\cdot x=\bot$ and thus
  $x\notin\M_\alpha$\,. The two properties of $F_\alpha$ follow.
\end{proof}

\subsection{Examples}
\label{sec:examples}

We introduce examples of concurrent systems, some of which will
serve as running examples in the remaining of the paper.

\subsubsection{A natural action}
\label{sec:natural-action}

The heap of pieces interpretation of traces
(\S~\ref{sec:heaps-pieces}) provides an example of a concurrent system,
which we informally describe.

Any heap of pieces, and thus any trace, can be seen as a labelled
partial order~\cite{viennot86,krattenthaler06}, the labels ranging
over the alphabet~$\Sigma$. As such, it has a number of maximal
pieces; let $M(x)$ be the labelled set of maximal pieces of a
heap~$x$. Let $\XS$ be the finite collection of all possible labelled
sets~$M(x)$, for $x$ ranging over~$\M$. Then $\M$ acts on $\XS$ by
$M\cdot x=M(y\cdot x)$, where $y$ is the flat heap that identifies
with~$M$.

Hence $M\cdot x$ represents the set of pieces that can be seen, from
above, when piling up the heap $x$ over~$M$.  In case $\M$ is the free
monoid~$\Sigma^*$, then $\XS$ identifies with $\Sigma$ and
$\alpha\cdot x$ is the last letter of the word~$\alpha x$. Hence this
action generalises in a natural way the Markov chain model, where
states and actions coincide.

\subsubsection{Petri nets}
\label{sec:petri-nets}

Petri nets, also called place/transition systems and which come in
several variants, are a well known models for concurrency
introduced in the 1960's and with as various applications as workflow
control, formal verification of logical properties of systems,
diagnosis and monitoring of discrete events systems, and performance
evaluation and queuing network theory on the applied mathematics
side~\cite{murata89,baccelli92}. 


Note that the stochastic models usually attached to Petri nets, as
in~\cite{baccelli92} for instance, randomise the execution sequences
of the net, not their equivalence classes modulo commutativity of
independent transitions, and differ thus radically from our model. By contrast, the encoding of $1$-safe Petri nets by trace monoid actions is explained in~\cite{abbes17}.

\subsubsection{The tip-top action}
\label{sec:tip-top-action}

Let $\M$ be a trace monoid and let $\XS=\C$, the set of cliques
of~$\M$ (\S~\ref{sec:cliques}). We define the tip-top action by means of the
partial action $\varphi:\C\times\M\to\C$ with:
\begin{gather*}
  \forall \gamma\in\C\quad\Sigma(\gamma)=\{a\in\Sigma\tq
  a\leq\gamma\vee a\parallel\gamma\}\\
\forall\gamma\in\C\quad\forall a\in\Sigma(\gamma)\quad \varphi(\gamma,a)=
\begin{cases}
\gamma\setminus\{a\},&\text{if $a\leq\gamma$}\\
\gamma\cdot a,&\text{if $a\parallel\gamma$}
\end{cases}
\end{gather*}

Let us verify that the tip-top action satisfies the condition of
Definition~\ref{def:2}. Let $\gamma\in\C$, and let $a,b\in\Sigma$ be
such that $a\in\Sigma(\gamma)$ and $a\parallel b$. We discuss
different cases, observing that in all cases, $b\neq a$ since
$a\parallel b$.
\begin{enumerate}\tightlist
\item Assume that $a\leq\gamma$. Hence
  $\varphi(\gamma,a)=\gamma\setminus\{a\}$.
  \begin{enumerate}
  \item If $b\in\Sigma(\gamma)$.
    \begin{enumerate}
    \item If $b\leq\gamma$, then
      $\varphi(\gamma,b)=\gamma\setminus\{b\}$. Since $a\neq b$, we
      also have $a\leq\gamma\setminus\{b\}$ and thus
      $a\in\Sigma(\varphi(\gamma,b))$ and
      $\varphi(\varphi(\gamma,b),a)=\gamma\setminus\{a,b\}$. For the
      same reason, $b\leq\gamma\setminus\{a\}$ and thus
      $b\in\Sigma(\varphi(\gamma,a))$ and
      $\varphi(\varphi(\gamma,a),b)=\gamma\setminus\{a,b\}=\varphi(\varphi(\gamma,b),a)$.

    \item If $b\parallel\gamma$, then $\varphi(\gamma,b)=\gamma\cdot
      b$. Hence $a\leq\varphi(\gamma,b)$ and thus
      $a\in\Sigma(\varphi(\gamma,b))$ and
      $\varphi(\varphi(\gamma,b),a)=\gamma\cdot b\setminus\{a\}$. We
      also have $b\parallel\gamma\setminus\{a\}$ thus
      $b\in\Sigma(\varphi(\gamma,a))$ and
      $\varphi(\varphi(\gamma,a),b)=(\gamma\setminus\{a\})\cdot
      b=\varphi(\varphi(\gamma,b),a)$.
    \end{enumerate}
  \item If $b\notin\Sigma(\gamma)$. Then $b$ is not parallel to
    $\gamma\setminus\{a\}$ otherwise, since $b\parallel a$, $b$~would
    be parallel to~$\gamma$, contradicting
    $b\notin\Sigma(\gamma)$. And $b\leq\gamma\setminus\{a\}$ does not
    hold either, otherwise $b\leq\gamma$ would hold, contradicting
    $b\notin\Sigma(\gamma)$. We conclude that
    $b\notin\Sigma(\varphi(\gamma,a))$, as expected.
  \end{enumerate}
\item Assume that $a\parallel \gamma$. Hence
  $\varphi(\gamma,a)=\gamma\cdot a$.
  \begin{enumerate}
  \item If $b\in\Sigma(\gamma)$.
    \begin{enumerate}
    \item If $b\leq\gamma$, then
      $\varphi(\gamma,b)=\gamma\setminus\{b\}$. Therefore
      $a\parallel\varphi(\gamma,b)$ thus
      $a\in\Sigma(\varphi(\gamma,b))$ and
      $\varphi(\varphi(\gamma,b),a)=(\gamma\setminus\{b\})\cdot a$. On
      the other hand, $b\leq\varphi(\gamma,a)$, thus
      $b\in\Sigma(\varphi(\gamma,a))$ and
      $\varphi(\varphi(\gamma,a),b)=(\gamma\cdot
      a)\setminus\{b\}=\varphi(\varphi(\gamma,b),a)$.

    \item If $b\parallel\gamma$, then $\varphi(\gamma,b)=\gamma\cdot
      b$, but $a\parallel b$ hence $a\parallel\varphi(\gamma,b)$, thus
      $a\in\Sigma(\varphi(\gamma,b))$ and
      $\varphi(\varphi(\gamma,b),a)=\gamma\cdot b\cdot a$\,. On the
      other hand, we also have $b\parallel\varphi(\gamma,a)$ since
      $a\parallel b$, thus $b\in\Sigma(\varphi(\gamma,a))$ and
      $\varphi(\varphi(\gamma,a),b)=\gamma\cdot a\cdot
      b=\varphi(\varphi(\gamma,b),a)$.
    \end{enumerate}

  \item If $b\notin\Sigma(\gamma)$. Then $b\leq\varphi(\gamma,a)$
    would imply $b\leq\gamma$ since $a\neq b$, contradicting
    $b\notin\Sigma(\gamma)$. And $b\parallel\varphi(\gamma,a)$ would
    imply $b\parallel\gamma$, contradicting
    $b\notin\Sigma(\gamma)$. Hence $b\notin\Sigma(\varphi(\gamma,a))$,
    as expected.
  \end{enumerate}
\end{enumerate}

\subsubsection{Domino tilings}
\label{sec:tilings}

Tilings of plane surfaces by $2\times1$ and $1\times2$ dominoes are an
``extensively studied and well-understood lattice model in statistical
physics and combinatorics''~\cite[\S~1.1]{romik12}. Here, we point out
that domino tilings and their flips provide a vast source of examples
of partial actions of trace monoids, and thus of concurrent systems; we will focus on elementary
examples only, using them for illustrative purposes.  

Given a plane surface for which there exists at least one domino
tiling, the flips rotate two contiguous dominoes and leave unchanged
the rest of the tiling~\cite{saldanha95}, as illustrated in
Figure~\ref{fig:qqpaxw}.

\begin{figure}[t]
  \centering
\begin{tabular}{ccc}
  \begin{picture}(40,40)
    \multiput(0,0)(0,20){3}{\line(1,0){40}}
\multiput(0,0)(40,0){2}{\line(0,1){40}}
\put(20,20){\circle*{4}}
\end{picture}
&
\begin{picture}(40,40)
  \put(0,20){ \xy<.08em,0em>:
(20,10),*{\txt{flip}}
 \ar@{<->} (0,00);(40,0) <0pt>
 \endxy
}
\end{picture}
&  \begin{picture}(40,40)
    \multiput(0,0)(0,40){2}{\line(1,0){40}}
\multiput(0,0)(20,0){3}{\line(0,1){40}}
\put(20,20){\circle*{4}}
  \end{picture}
\end{tabular}
\caption{\small\textsl{Illustrating the action of a flip. The dot represents the
      centre of the rotation.}}
  \label{fig:qqpaxw}
\end{figure}
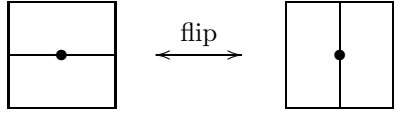

\subsubsection{The Rabati tilings}
\label{sec:rabati-tiling}

We define the $n$-Rabati strip, with $n\geq2$ an integer, as a strip
of size $n\times2$, as depicted on Figure~\ref{fig:rabatione} for
$n=7$. The in-line version is the one depicted on the figure; the
circled version corresponds to the same strip, where the right most
side is identified with the left most side.

\begin{figure}[b]
\begin{align*}
  \begin{picture}(140,40)
  \multiput(0,0)(0,40){2}{\line(1,0){140}}
  \multiput(0,0)(20,0){8}{\line(0,1){40}}
  \multiput(20,20)(20,0){6}{\circle*{4}}
  \put(23,22){$a_1$}
  \put(43,22){$a_2$}
  \put(63,22){$a_3$}
  \put(83,22){$a_4$}
  \put(103,22){$a_5$}
  \put(123,22){$a_6$}
\end{picture}
&&
  \begin{picture}(140,40)
  \multiput(0,0)(0,40){2}{\line(1,0){140}}
  \multiput(0,0)(40,0){2}{\line(0,1){40}}
  \multiput(60,0)(20,0){3}{\line(0,1){40}}
  \put(100,0){\line(0,1){40}}
  \multiput(0,20)(100,0){2}{\line(1,0){40}}
  \put(140,0){\line(0,1){40}}
  \multiput(20,20)(40,0){2}{\circle*{4}}
  \multiput(80,20)(40,0){2}{\circle*{4}}
  \put(23,22){$a_1$}
  \put(63,22){$a_3$}
  \put(83,22){$a_4$}
  \put(123,22){$a_6$}
\end{picture}
\end{align*}
\caption{\small\textsl{The in-line $n$-Rabati strip with $n=7$. The tiling
    on the left corresponds to the empty clique. The tiling on the
    right correspond to the clique~$a_1\cdot a_6$\,.}}
  \label{fig:rabatione}
\end{figure}
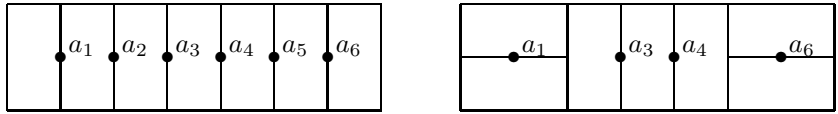

Consider the in-line version of the Rabati strip first. Then associate
the trace monoid with $n-1$ generators, say
$\Sigma=\{a_1,\ldots,a_{n-1}\}$, and independence relation $I$ defined by:
\begin{gather*}
  I=\{(a_i,a_j)\tq |i-j|\geq2\}\,.
\end{gather*}

For the circle version of the Rabati strip, we consider the trace
monoid with $n$ generators, $\Sigma=\{a_0,\ldots,a_{n-1}\}$, and
independence relation:
\begin{gather*}
  I=\{(a_i,a_j)\tq (|i-j|\mod n)\geq2\}
\end{gather*}

In both cases, the following facts are elementary (see
Figure~\ref{fig:rabatione}):
\begin{enumerate}\tightlist
\item The domino tilings of the Rabati strip are in bijection with the
  cliques of the associated trace monoid. The empty clique corresponds
  to the tiling where all dominoes are in vertical position.
\item Let $\gamma\in\C$ correspond to some tiling of the strip. The
  flip centred on the $i^{\text{th}}$~position is enabled if and only
  if the tip-top action $\varphi(\gamma,a_i)$ is enabled; and the
  result $\varphi(\gamma,a_i)$ of the tip-top action corresponds to
  the tiling resulting from the action of this flip on the
  tiling~$\gamma$.
\end{enumerate}

Hence, for tilings of Rabati strips, the action of flips corresponds
exactly to the tip-top action of the associated trace monoid. 

For tilings of general plane surfaces, we can still define a trace
monoid associated to the action of flips, where each flip corresponds
to a generator of the monoid, and commutativity of generators
corresponds to flips that operate on disjoint pairs of dominoes. The
combinatorial action of flips is encoded in the action of this trace
monoid on the set of tilings. But this action does not correspond in
general to the tip-top action of the trace monoid---it is, in general,
much more complex.

\subsection{Sub-concurrent systems and irreducibility}
\label{sec:sub-acti-irred}

We borrow some usual terminology and notions from sub-shift theory
\cite{seneta81,lind95} to qualify states and define irreducibility
classes of concurrent systems. Throughout this subsection, we fix a concurrent system $(\M,\XS,\bot)$.

\subsubsection{Communicating states and essential states}
\label{sec:essential-states}

A state $\alpha\neq\bot$ \emph{leads} to a state~$\beta$, which we denote by
$\alpha\leads\beta$, if there exists a non empty trace
$x\in\M_\alpha\setminus\{1\}$ such that $\alpha\cdot x=\beta$ (and
thus $\beta\neq\bot$). Two states $\alpha$ and $\beta$ \emph{communicate} if
both $\alpha\leads\beta$ and $\beta\leads\alpha$ hold, which we denote
by $\alpha\lleads\beta$. A state $\alpha$ is \emph{essential} if
$\alpha\neq\bot$ and holds: $\forall\beta\in\XS\quad
\alpha\leads\beta$ implies $\alpha\lleads\beta$. 

Note that, by the monoid action property, the relation $\alpha\leads\beta$ is
transitive on~$\XS\setminus\{\bot\}$. Also, since
$\Sigma(\alpha)\neq\emptyset$ according to Definition~\ref{def:2},
every state leads to at least one state. 



\subsubsection{Sub-concurrent systems and irreducible concurrent systems}
\label{sec:sub-actions}

\begin{definition}
Let $(\M,\XS,\bot)$ be a concurrent system with $\M=\M(\Sigma,I)$. Let $Y$ be a non empty subset of~$X\setminus\{\bot\}$, satisfying:
\begin{gather}
  \label{eq:51}
  \forall\alpha\in Y\quad\forall a\in\Sigma(\alpha)\quad \alpha\cdot a\in Y.
\end{gather}
Then the restriction of\/ $\XS\times\M\to\XS$ to $(Y\cup\{\bot\})\times\M\to
Y\cup\{\bot\}$ defines a concurrent system $(\M,Y\cup\{\bot\},\bot)$. Every concurrent system of this form is called a \emph{sub-concurrent system} of\/~$(\M,\XS,\bot)$.

The concurrent system $(\M,\XS,\bot)$ is said to be \emph{irreducible} if it has no other sub-concurrent system than itself.
\end{definition}

Sub-concurrent systems correspond  to non empty subsets $Y$ of $\XS\setminus\{\bot\}$ closed under
the communicating relation between states. Furthermore, states of an irreducible sub-concurrent system are essentials, both in the original concurrent system and in the sub-concurrent system.


Considering the decomposition of the graph $(X,\leads)$ into its set of strongly connected components yields the following elementary result.

\begin{proposition}
  \label{prop:5}
  The set of essential states of a concurrent system $(\M,\XS,\bot)$ is non empty and defines a sub-action. For
  each essential state~$\alpha$, the set $X_\alpha=\{\beta\in\XS\tq \alpha\leads\beta\}$
  defines an irreducible sub-concurrent system of $(\M,\XS,\bot)$. Every irreducible sub-concurrent system is of this form.
\end{proposition}

\subsubsection{Example}
\label{sec:example}

The tip-top action $\varphi:\C\times\M\to\C$ of a trace monoid over
its set of cliques (\S~\ref{sec:tip-top-action}) is irreducible, and
thus every clique is an essential state. Indeed, if $\gamma\neq1$,
then $\varphi(\gamma,\gamma)=1$; and for $\gamma=1$, pick any letter
$a\in\Sigma$, then
$\varphi(\varphi(1,a),a)=\varphi(a,a)=1$. Henceforth every clique
leads to the empty clique. And similarly it is easy to see that the
empty clique leads to every clique. We deduce $\gamma\lleads\gamma'$
for all cliques~$\gamma,\gamma'$\,, and thus the action is
irreducible.



\section{Markov measures}
\label{sec:markov-measures}

We now come to our main object of study, Markov measures associated
with concurrent systems. In this section, we examine Markov
measures associated with trace monoid actions.  Next section is devoted to
additional topics of study when the monoid action actually originates
from a partial action.

\subsection{Markov measures and fibred valuations}
\label{sec:mark-meas-fibr}

\subsubsection{Definitions}
\label{sec:definition-2}

We recall that the boundary $\BM$ of a trace monoid~$\M$
(\S~\ref{sec:infinite-traces}) is equipped with its Borel \slgb\
(\S~\ref{sec:slgb}).

\begin{definition}
  \label{def:4}
  Let $\XS\times\M\to\XS$ be a monoid action of a trace monoid $\M$
  over a finite set~$X$. A Markov measure associated with this action
  is a family\/ $\pr=(\pr_\alpha)_{\alpha\in\XS}$\,, where each\/
  $\pr_\alpha$ is a probability measure on the boundary\/~$\BM$,
  satisfying the following chain rule:
\begin{gather}
  \label{eq:4}
\forall\alpha\in\XS\quad\forall x,y\in\M\quad\pr_\alpha(\up(x\cdot
y))=\pr_\alpha(\up x)\pr_{\alpha\cdot x}(\up y)\,.
\end{gather}

A fibred valuation associated with this action is a family
$F=(f_\alpha)_{\alpha\in\XS}$\,, where each $f_\alpha:\M\to\bbR$ is a
real-valued mapping, with the  following property:
\begin{gather}
  \label{eq:5}
\forall\alpha\in\XS\quad f_\alpha(1)=1\\
 \label{eq:6}
\forall\alpha\in\XS\quad\forall x,y\in\M\quad f_\alpha(x\cdot
y)=f_\alpha(x) f_{\alpha\cdot x}(y)
\end{gather}

If\/ $\pr=(\pr_\alpha)_{\alpha\in\XS}$ is a Markov measure, the family
$F=(f_\alpha)_{\alpha\in\XS}$ defined by:
\begin{gather}
  \label{eq:7}
\forall\alpha\in\XS\quad\forall x\in\M\quad f_\alpha(x)=\pr_\alpha(\up x)
\end{gather}
is called the fibred valuation associated with\/~$\pr$\,. The support
valuation of\/ $\pr$ is the family $S=(s_\alpha)_{\alpha\in\XS}$ of
mappings $s_\alpha:\M\to\{0,1\}$ defined by:
\begin{gather}
  \label{eq:8}
\forall\alpha\in\XS\quad\forall x\in\M\quad s_\alpha(x)=
\begin{cases}
1,&\text{if $\pr_\alpha(\up x)>0$,}\\
0,&\text{if $\pr_\alpha(\up x)=0$.}
\end{cases}
\end{gather}
\end{definition}

\subsubsection{Elementary remarks}
\label{sec:elementary-remarks}

Bernoulli measures on~$\BM$ (\S~\ref{sec:bernoulli-measures})
correspond to Markov measures $\pr=(\pr_\alpha)_{\alpha\in\XS}$ such
that $\pr_\alpha$ is independent of~$\alpha$. In particular, Markov
measures exist. 

If $\M=\Sigma^*$ is the free monoid with its natural right action
on~$\Sigma$, then Markov measures correspond to Markov chains
on~$\Sigma$. Bernoulli measures correspond to Bernoulli sequences with
values in~$\Sigma$.

Clearly, the support valuation of a Markov measure is a fibred
valuation.

Since the family of elementary cylinders forms a $\pi$-system
(\S~\ref{sec:slgb}) of the Borel \slgb\ on~$\BM$, a Markov measure
$\pr=(\pr_\alpha)_{\alpha\in\XS}$ is entirely characterised by the
countable family of non negative numbers $\{\pr_\alpha(\up x)\tq
\alpha\in\XS,\ x\in\M\}$. In turn, by the chain rule~(\ref{eq:4}),
these are entirely determined by the finite family of non negative
numbers $\{f_\alpha(a)\tq\alpha\in\XS,\ a\in\Sigma\}$. Hence,
characterising Markov measures consists in finding adequate
normalisation conditions on finite families of non negative numbers of
the form~$(\lambda_\alpha(a))_{(\alpha,a)\in\XS\times\Sigma}$\,.

\subsubsection{Fibred valuations}
\label{sec:fibred-valuations}

A first task consists in elucidating the---rather simple---structure
of fibred valuations.

\begin{proposition}
\label{prop:6}  
Fibred valuations are in bijection with families
$(\lambda_\alpha(a))_{(\alpha,a)\in\XS\times\Sigma}$ of reals numbers
satisfying the following series of equations, that we call the
concurrency equations:
\begin{gather}
  \label{eq:9}
  \forall\alpha\in\XS\quad\forall a,b\in\Sigma\quad a\parallel
  b\implies \lambda_\alpha(a)\lambda_{\alpha\cdot
    a}(b)=\lambda_\alpha(b)\lambda_{\alpha\cdot b}(a)
\end{gather}
The bijection associates to a fibred valuation
$(f_\alpha)_{\alpha\in\XS}$ the family of
reals~$(f_\alpha(a))_{(\alpha,a)\in\XS\times\Sigma}$\,.
\end{proposition}

\begin{proof}
  If $(f_\alpha)_{\alpha\in\XS}$ is a fibred valuation, then the
  family of reals defined by $\lambda_\alpha(a)=f_\alpha(a)$
  satisfies the concurrency equations, since both the left and the
  right members of the equation in~(\ref{eq:9}) represent
  $f_\alpha(a\cdot b)$, when $a\parallel b$. 

  Conversely, let $(\lambda_\alpha(a))_{(\alpha,a)\in\XS\times\Sigma}$
  be a family of reals satisfying the concurrency equations. We
  show the existence of a fibred valuation
  $F=(f_\alpha)_{\alpha\in\XS}$ satisfying
  $f_\alpha(a)=\lambda_\alpha(a)$ for all $\alpha\in\XS$ and
  $a\in\Sigma$. 

  Define for each $\alpha\in\XS$ and each word $x=a_1\ldots a_n$
  of~$\Sigma^*$:
\begin{gather}
\label{eq:10}
f_\alpha(x)=\lambda_\alpha(a_1) \lambda_{\alpha\cdot a_1}(a_2)\cdots
\lambda_{\alpha\cdot a_1\cdot\ldots\cdot a_{n-1}}(a_n)\,,
\end{gather}
with $f_\alpha(1)=1$ by convention.  Then the concurrency equations
imply that $f_\alpha(x)=f_\alpha(y)$ for any word $y$ which is in
immediate equivalence with~$x$ (\S~\ref{sec:immed-equiv}). Therefore
$f_\alpha$ factorises through a mapping, still denoted
$f_\alpha:\M\to\bbR$ such that $f_\alpha(a)=\lambda_\alpha(a)$ for all
$a\in\Sigma$. The definition~(\ref{eq:10}) implies the validity of
the chain rule~(\ref{eq:6}), whence the sought fibred valuation
$F=(f_\alpha)_{\alpha\in\XS}$\,. 
\end{proof}

\begin{corollary}
  \label{cor:1}
  If\/ $(\pr_\alpha)_{\alpha\in\XS}$ is a Markov measure, then the
  numbers $\lambda_\alpha(a)=\pr_\alpha(\up a)$, for $\alpha$ ranging
  over $\XS$ and $a$ ranging over~$\Sigma$, satisfy the concurrency
  equations\/~\eqref{eq:9}.
\end{corollary}

\subsection{Characterisation of Markov measures}
\label{sec:irred-mark-meas}

Our aim is to give necessary and sufficient conditions for a family
$(\lambda_\alpha(a))_{(\alpha,a)\in\XS\times\Sigma}$ of real numbers
to correspond to the family
$(f_\alpha(a))_{(\alpha,a)\in\XS\times\Sigma}$ obtained from the
fibred valuation of a Markov measure. Theorem~\ref{thr:1} below gives
necessary conditions, postponing their sufficiency to next subsection.

\subsubsection{M\"obius fibred valuations}
\label{sec:char-irred-mark}

The following definition generalises to fibred valuations the notion
of M\"obius valuation recalled in~\S~\ref{sec:mobius-valuations}.

\begin{definition}
  \label{def:6}
Let $(\M,\XS)$ be a trace monoid action, and let $F=(f_\alpha)_{\alpha\in\XS}$ be a fibred valuation. For each
  $\alpha\in\XS$, let $h_\alpha:\C\to\bbR$ be the M\"obius transform\/
  {\normalfont(\S~\ref{sec:mobius-transform})} of
  $f_\alpha:\C\to\bbR$. Then  $F$ is \emph{M\"obius fibred valuation} if it
  satisfies the following conditions, for all $\alpha\in\XS$:
\begin{align}
\label{eq:21}
  h_\alpha(1)&=0\,,\\
\label{eq:22}
\forall\gamma\in\Cstar\quad
 h_\alpha(\gamma)&\geq0\,.
\end{align}
\end{definition}

Remark that, as a consequence of the second M\"obius inversion formula
(\S~\ref{sec:second-mobi-invers}), a M\"obius valuation is necessarily
non negative.

\subsubsection{Identification theorem}
\label{sec:main-theorem}

The purpose of the following theorem is twofold. First, it provides
necessary conditions on families of numbers that generate a Markov
measure. Second, it shows that the random process of
``states-and-cliques'' (defined in the statement) has the structure of
a Markov chain, of which the initial measure and the transition matrix
are entirely determined by the Markov measure. In both aspects, it
generalises the corresponding results for Bernoulli measures
(\S~\ref{sec:caract-bern-meas} and~\S~\ref{sec:markov-chain-cliques}).

\begin{theorem}
  \label{thr:1}
  Let\/ $\pr=(\pr_\alpha)_{\alpha\in\XS}$ be a Markov measure associated to a trace monoid action~$(\M,\XS)$. Then:
  \begin{enumerate}\tightlist
  \item\label{item:7} The fibred valuation associated with\/ $\pr$ is
    M\"obius.
  \item\label{item:8} Let $(C_i)_{i\geq1}$ and $(Y_i)_{i\geq0}$ be the
    two random sequences introduced
    in\/~{\normalfont\S~\ref{sec:rand-decomp-infin}}, and for each
    integer $i\geq0$, let $X_i=\alpha\cdot Y_{i}$\,. Then, under the
    probability\/~$\pr_{\alpha_0}$\,, the sequence
    $(X_{i-1},C_i)_{i\geq1}$ is a homogeneous Markov chain with values
    in~$\XS\times\Cstar$.
  \item\label{item:9} The initial measure of the Markov chain
    $(X_{i-1},C_i)_{i\geq1}$ is $\delta_{\{\alpha_0\}}\otimes
    (h_{\alpha_0}\rest\Cstar)$\,, where
    $h_{\alpha_0}\rest\Cstar:\Cstar\to\bbR$ denotes the restriction of
    $h_{\alpha_0}$ to non empty cliques, and $h_{\alpha_0}:\C\to\bbR$
    is the M\"obius transform of the fibred valuation associated
    with\/~$\pr$.
  \item\label{item:10} The transition matrix of the chain is
    independent of~$\alpha_0$\,. It is given by:
    \begin{gather}
      \label{eq:12}
      P_{(\alpha,\gamma),(\alpha',\gamma')}=\un_{\{\alpha'=\alpha\cdot\gamma\}}\un_{\{\gamma\to\gamma'\}}
\frac{h_{\alpha'}(\gamma')}{g_{\alpha'}(\gamma)}\,,
    \end{gather}
    where $\gamma\to\gamma'$ denotes the Cartier-Foata relation
    between cliques {\normalfont(\S~\ref{sec:cart-foata-relat})}, and
    $g_{\alpha'}:\Cstar\to\bbR$ is the non negative function defined by:
    \begin{gather}
\label{eq:13}
\forall\gamma\in\Cstar\quad
g_{\alpha'}(\gamma)=\sum_{\gamma'\in\Cstar\tq\gamma\to\gamma'}h_{\alpha'}(\gamma')\,.
    \end{gather}
    If $g_{\alpha'}(\gamma)=0$, then the expression\/~\eqref{eq:12} is
    replaced by:
\begin{gather*}
  P_{(\alpha,\gamma),(\alpha',\gamma')}=0\,.
\end{gather*}
  \end{enumerate}
\end{theorem}

\subsubsection{Remark on the transition matrix}
\label{sec:remark-trans-matr}

The matrix $P$ described above might have some lines entirely filled
with zeros. This happens for the lines indexed by pairs
$(\alpha,\gamma)$ where $\gamma$ is such that
$g_{\alpha\cdot\gamma}(\gamma)=0$\,.

Although formally forbidden for stochastic matrices, this is actually
of little inconvenience since these lines correspond to states that
will never be reached.

\subsubsection{Strategy of proof}
\label{sec:outline-proof}

We first establish the two following key lemmas. All the statements of
Theorem~\ref{thr:1} derive from them.

The proof follows closely the same line as the proof of the
corresponding results for Bernoulli measures
(\S~\ref{sec:caract-bern-meas} and~\S\ref{sec:markov-chain-cliques}),
with the specific issue here consisting in correctly dealing with the
random current state.

\begin{lemma}
  \label{lem:2}
  Let\/ $(\pr_\alpha)_{\alpha\in\XS}$ be a Markov measure associated
to a trace monoid action~$(\M,\XS)$, and with fibred
  valuation~$(f_\alpha)_{\alpha\in\XS}$\,. For each $\alpha\in\XS$,
  let $h_\alpha:\M\to\bbR$ be the graded M\"obius transform
  {\normalfont(\S~\ref{sec:grad-mobi-transf})} of~$f_\alpha(\cdot)$.
  Let $n\geq1$ be an integer, and let $x\in\M$ be a trace of height
  $\height(x)=n$. Then:
    \begin{gather}
      \label{eq:11}
      \pr_\alpha(Y_n=x)=h_\alpha(x)\,.
    \end{gather}
\end{lemma}

\begin{proof}
  Let $x$ and $n$ be as in the statement. The formula~(\ref{eq:16})
  decomposes the elementary cylinder $\up x$ as a disjoint union, whence:
  \begin{align*}
    \pr_\alpha(\up x)=\sum_{z\in\M(x)}\pr_\alpha(Y_n=z)\,.
  \end{align*}
  We also have $\pr_\alpha(\up x)=f_\alpha(x)$ by definition
  of~$f_\alpha$\,. Hence, by the reciprocal of the second M\"obius
  inversion formula (\S~\ref{sec:second-mobi-invers}), we deduce that
  the two functions $f_\alpha$ and
  $x\in\M\mapsto\pr_\alpha(Y_{\height(x)}=x)$ are related by the
  graded M\"obius formula, whence~(\ref{eq:11}).
\end{proof}

\begin{lemma}
  \label{lem:1}
  Let $(\M,\XS)$ be a trace monoid action, and let $F=(f_\alpha)_{\alpha\in\XS}$ be a fibred valuation. For each
  $\alpha\in\XS$, let $h_\alpha:\C\to\bbR$ be the M\"obius transform
  of~$f_\alpha$\,, and let $g_\alpha:\Cstar\to\bbR$ be defined as
  in\/~\eqref{eq:13}. Assume that $h_\alpha(1)=0$ for all
  $\alpha\in\XS$ (in particular, this holds if $F$ is M\"obius).

  Then, for all $\beta\in\XS$ and for all $\gamma\in\Cstar$, holds:
\begin{align}
  \label{eq:14}
f_{\beta}(\gamma)g_\alpha(\gamma)&=h_{\beta}(\gamma)\,,\qquad
\text{for }\alpha=\beta\cdot\gamma\,.
\end{align}
\end{lemma}

\begin{proof}
  Let $\beta\in\XS$ and $\gamma\in\Cstar$, and put
  $\alpha=\beta\cdot\gamma$. We compute as follows:
\begin{align*}
  g_\alpha(\gamma)&=\sum_{\gamma'\in\Cstar\tq\gamma\to\gamma'}h_\alpha(\gamma')\\
&=\sum_{\gamma'\in\Cstar\tq\gamma\to\gamma'}\Bigl(\sum_{\gamma''\in\C\tq\gamma''\geq\gamma'}(-1)^{|\gamma''|-|\gamma'|}f_\alpha(\gamma'')\Bigr)\\
&=\sum_{\gamma''\in\C}(-1)^{|\gamma''|}f_\alpha(\gamma'')B(\gamma,\gamma'')\\
\intertext{with}
B(\gamma,\gamma'')&=
\sum_{\gamma'\in\Cstar}(-1)^{|\gamma'|}\un_{\{\gamma\to\gamma'\}}\un_{\{\gamma'\leq\gamma''\}}\,.
\end{align*}

For $\gamma,\gamma''\in\C$, a clique $\gamma'\in\Cstar$ satisfies
$\gamma\to\gamma'$ and $\gamma'\leq\gamma''$ if and only if
$\gamma''\leq\delta$, where $\delta=\{a\in\gamma''\tq\gamma\to
a\}$. Therefore, adding and subtracting the empty clique in the sum
defining $B(\gamma,\gamma'')$ and using the binomial formula, we find:
\begin{align*}
  B(\gamma,\gamma')&=\sum_{\gamma'\in\C\tq\gamma'\leq\delta}(-1)^{|\gamma'|}-1
=\un_{\{\delta=1\}}-1
=-\un_{\{\neg(\gamma''\parallel\gamma)\}}\,.
\end{align*}
We obtain thus the following expression for~$g_\alpha(\gamma)$:
\begin{gather}
  \label{eq:18}
g_\alpha(\gamma)=-\sum_{\gamma''\in\C\tq\neg(\gamma''\parallel\gamma)}(-1)^{|\gamma''|}f_\alpha(\gamma'')\,.
\end{gather}
  
Writing down the hypothesis $h_\alpha(1)=0$ yields, by the very
definition of the M\"obius transform~$h_\alpha$:
\begin{gather*}
  \sum_{\gamma''\in\C}(-1)^{|\gamma''|}f_\alpha(\gamma'')=0\,.
\end{gather*}
Separating this sum according to those cliques $\gamma''$ which are on
the one hand, and which are not on the other hand, parallel
to~$\gamma$, and by virtue of~(\ref{eq:18}), we obtain:
\begin{gather}
\label{eq:19}
  g_\alpha(\gamma)=\sum_{\gamma''\in\C\tq\gamma''\parallel\gamma}(-1)^{|\gamma''|}f_\alpha(\gamma'')\,.
\end{gather}

Since $\alpha=\beta\cdot\gamma$ on the one hand, and by the chain rule
for fibred valuations on the other hand, we have:
\begin{gather*}
f_\beta(\gamma\cdot\gamma'')=f_\beta(\gamma)f_\alpha(\gamma'')\,.
\end{gather*}

Therefore, multiplying both sides of~(\ref{eq:19}) by
$f_\beta(\gamma)$ yields:
\begin{align*}
  f_\beta(\gamma)g_\alpha(\gamma)&=\sum_{\gamma''\parallel\gamma}(-1)^{|\gamma''|}f_\beta(\gamma\cdot\gamma'')\\
&=\sum_{\delta\in\C\tq\delta\geq\gamma}(-1)^{|\delta|-|\gamma|}f_\beta(\delta)\\
&=h_\beta(\gamma)\,,
\end{align*}
which was to be proved.
\end{proof}

\subsubsection{Proof of Theorem~{\normalfont\ref{thr:1}}}
\label{sec:proof-theorem}

Let us show that the fibred valuation $F=(f_\alpha)_{\alpha\in\XS}$
associated with the Markov measure $\pr$ is M\"obius. According
to Lemma~\ref{lem:2}, and since $C_1=Y_1$\,, we have:
\begin{gather*}
  \forall\gamma\in\Cstar\quad\pr_\alpha(C_1=\gamma)=h_\alpha(\gamma)\,.
\end{gather*}

It follows at once that $h_\alpha(\gamma)\geq0$ for all
$\gamma\in\Cstar$, which proves the property~(\ref{eq:22}). And
applying the total probability law to the first clique~$C_1$ yields:
\begin{align*}
  \sum_{\gamma\in\Cstar}h_\alpha(\gamma)=1\,.
\end{align*}
But we also have, by the M\"obius inversion formula~(\ref{eq:17}):
\begin{align*}
  f_\alpha(1)&=h_\alpha(1)+\sum_{\gamma\in\Cstar}h_\alpha(\gamma)\\
1&=h_\alpha(1)+1\,,
\end{align*}
whence $h_\alpha(1)=0$, which proves the property~(\ref{eq:21}).
Henceforth we have proved that $F$ is M\"obius and that the law of
$(X_0,C_1)$ under $\pr_\alpha$ is indeed given by
$\delta_{\{\alpha\}}\otimes(h_\alpha\rest\Cstar)$\,. This proves
points~\ref{item:7} and~\ref{item:9} of the theorem.

Let $P$ be the square matrix indiced by $(\XS\times\Cstar)^2$ and
defined by~(\ref{eq:12}), with the restriction that
$P_{(\alpha,\gamma),(\alpha',\gamma')}=0$ whenever
$g_\alpha(\gamma)=0$. We shall prove that, for any sequence
$(x_0,c_1),\ldots,(x_{n-1},c_n)$ in $\XS\times\Cstar$ with $n\geq1$,
if $p_n$ denotes the probability
\begin{gather*}
p_n=  \pr_\alpha\bigl(
(X_0,C_1)=(x_0,c_1),\ldots,(X_{n-1},C_n)=(x_{n-1},c_n)\bigr)
\end{gather*}
and if $q_n$ denotes the result of the chain rule:
\begin{gather*}
  q_n=\un_{\{\alpha=x_0\}}h_{x_0}(c_1)P_{(x_0,c_1),(x_1,c_2)}\dots P_{(x_{n-2},c_{n-1}),(x_{n-1},c_n)}
\end{gather*}
then $p_n=q_n$\,. This will prove the statements~\ref{item:8}
and~\ref{item:10} and complete the proof Theorem~\ref{thr:1}. 

First, it is clear that $p_n=q_n=0$ if any of the following statements
does not hold:
\begin{gather}
\label{eq:20}
  x_0=\alpha\,,\\
\label{eq:23}
\forall i\in\{0,\ldots,n-2\}\quad x_i\cdot c_{i+1}=x_{i+1}\,,\\
\label{eq:24}
\forall i\in\{1,\ldots,n-1\}\quad c_i\to c_{i+1}\,.
\end{gather}

Hence, assuming that all the above statements
(\ref{eq:20})--(\ref{eq:24}) hold, we have:
\begin{align*}
  p_n&=\pr_\alpha(Y_n=c_1\cdot\ldots\cdot c_n)\\
&=h_\alpha(c_1\cdot\ldots\cdot c_n)&&\text{by Lemma~\ref{lem:2}}
\end{align*}

Defining $y_j=c_1\cdot\ldots\cdot c_j$ for all
$j\in\{0,\ldots,n\}$\,, the definition of the graded M\"obius
transform (\S~\ref{sec:grad-mobi-transf}) together with the chain rule
relation for fibred valuations yield:
\begin{gather}
  \label{eq:25}
p_n=f_\alpha(y_{n-1})h_{x_{n-1}}(c_n)
\end{gather}

The above implies in particular the following inequality, valid for
any integer $j\leq n$:
\begin{align*}
  p_n&\leq\pr_\alpha(Y_j=c_1\cdot\ldots\cdot c_j)\\
&\leq f_\alpha(y_{j-1})h_{x_{j-1}}(c_j)
\end{align*}

Therefore,  if $g_{x_{j-1}}(c_{j-1})=0$ for some integer
$j\in\{2,\ldots,n\}$, then on the one hand, 
$h_{x_{j-1}}(\gamma)=0$ for all $\gamma\in\Cstar$ such that
$c_{j-1}\to\gamma$ holds, and in particular $h_{x_{j-1}}(c_j)=0$, and thus
$p_n=0$. But $q_n=0$ also on the other hand, and thus $p_n=q_n$\,. 

It remains thus to examine the most interesting case, where all
statements (\ref{eq:20})--(\ref{eq:24}) hold and
$g_{x_{j-1}}(c_j)\neq0$ for all $j\in\{2,\ldots,n\}$. In this case, we
apply the result of Lemma~\ref{lem:1} and we divide by
$g_{x_{j-1}}(c_j)$ to obtain, since $x_j=x_{j-1}\cdot c_j$:
\begin{gather}
\label{eq:26}
\forall j\in\{1,\ldots,n-1\}\quad f_{x_{j-1}}(c_j)=\frac{h_{x_{j-1}}(c_j)}{g_{x_j}(c_j)}
\end{gather}

We evaluate~$q_n$:
\begin{align*}
  q_n&=h_{x_0}(c_1)\frac{h_{x_1}(c_2)}{g_{x_1}(c_1)}\dots\frac{h_{x_{n-1}}(c_n)}{g_{x_{n-1}}(c_{n-1})}\\
&=\frac{h_{x_0}(c_1)}{g_{x_1}(c_1)}\frac{h_{x_1}(c_2)}{g_{x_2}(c_2)}\dots 
\frac{h_{x_{n-2}}(c_{n-1})}{g_{x_{n-1}}(c_{n-1})}h_{x_{n-1}}(c_n)\\
&=f_{x_0}(c_1)f_{x_1}(c_2)\dots f_{x_{n-2}}(c_{n-1})
h_{x_{n-1}}(c_n)&&\text{by~(\ref{eq:26})}\\
&=f_\alpha(c_1\cdot\ldots\cdot c_{n-1})h_{x_{n-1}}(c_n)&&\text{by the
  chain rule}\\
&=p_n&&\text{by~(\ref{eq:25})}
\end{align*}
This completes the proof of Theorem~\ref{thr:1}.\qed

\subsection{Realisation of Markov measures}
\label{sec:markov-chain-states}

\subsubsection{A reciprocal to Theorem~{\normalfont\ref{thr:1}}}
\label{sec:reciprocal-theorem}

A striking result of Theorem~\ref{thr:1} is that the fibred valuation
associated with a Markov measure is necessarily M\"obius. The next
result shows that the converse is true. Furthermore, it provides a
realisation of the associated Markov measure by means of the Markov
chain of states-and-cliques.

The effectiveness of the M\"obius conditions for fibred valuations,
and thus for the parameters of Markov measures, is discussed below
in~\S~\ref{sec:effect-mobi-cond}.

\begin{theorem}
  \label{thr:2}
Let $(\M,\XS)$ be a trace monoid action, and let $F=(f_\alpha)_{\alpha\in\XS}$ be a M\"obius fibred valuation. Then there exists a unique Markov measure\/
$\pr=(\pr_\alpha)_{\alpha\in\XS}$ such that
$f_\alpha(x)=\pr_\alpha(\up x)$ for all $\alpha\in\XS$ and all
$x\in\M$. 

Let $P$ be the transition matrix on $\XS\times\Cstar$ defined
in\/~\eqref{eq:12}. For each $\alpha\in\XS$, let
$(X_{i-1},C_i)_{i\geq1}$ be a Markov chain with transition matrix~$P$
and with initial measure
$\delta_{\{\alpha\}}\otimes(h_\alpha\rest\Cstar)$, defined on the
canonical probability space $(\Omega,\GGG,\prq_\alpha)$. Then
$\pr_\alpha$ is the law of the random variable $\xi:\Omega\to\BM$
defined by:
\begin{gather*}
  \xi=\bigvee_{n\geq1}(C_1\cdot\ldots\cdot C_n)\,,
\end{gather*}
under the probability\/~$\prq_\alpha$\,. 
\end{theorem}

\begin{proof}
  We first need to check that the Markov chain in the statement is
  well defined, that is to say, that
  $\delta_{\{\alpha\}}\otimes(h_\alpha\rest\Cstar)$ is a probability
  distribution on $\XS\times\Cstar$, and that $P$ is a stochastic
  matrix. The second point is obvious by construction. For the first
  point, indeed $h_\alpha\rest\Cstar$ is a probability distribution on
  $\Cstar$ since, on the one hand, $h_\alpha\geq0$ by definition of
  $F$ being M\"obius, and on the other hand, by the second M\"obius
  inversion formula (\S~\ref{sec:second-mobi-invers}):
  \begin{align*}
    \sum_{\gamma\in\Cstar}h_\alpha(\gamma)&=\sum_{\gamma\in\C}h_\alpha(\gamma)&&\text{since
      $h_\alpha(1)=0$}\\
&=f_\alpha(1)=1.
  \end{align*}

  For each $\alpha\in\XS$, let $\pr_\alpha$ be the probability measure
  on $\BM$ defined as in the statement as the law of
  $\bigvee_{n\geq1}(C_1\cdot\ldots\cdot C_n)$\,,
  under~$\prq_\alpha$\,. Let $x\in\M$ be a trace. The decomposition
  from~\S~\ref{sec:compl-elem-cylind} of the elementary cylinder~$\up
  x$ yields:
\begin{align}
\label{eq:43}
  \pr_\alpha(\up x)=\sum_{z\in\M(x)}\prq_\alpha(C_1\cdot\ldots\cdot C_{\height(x)}=z)\,.
\end{align}

Let $n=\height(x)$, and $z\in\M$ be a trace of height
$\height(z)=n$. Then $z$ has a Cartier-Foata decomposition of the form
$c_1\to\ldots\to c_n$\,. We set $y_j=c_1\cdot\ldots\cdot c_j$ and
$x_j=\alpha\cdot y_j$ for all integers $j\in\{0,\ldots,n\}$, and:
\begin{align*}
  q_n&=\prq_\alpha(C_1=c_1,\ldots,C_n=c_n)\,.
\end{align*}

Then we
have:
\begin{align*}
q_n  &=\prq_\alpha(X_0=\alpha,C_1=c_1,\ldots,X_{n-1}=\alpha_{n-1},C_n=c_n)\\
&=h_\alpha(c_1)\frac{h_{x_1}(c_2)}{g_{x_1}(c_1)}\ldots\frac{h_{x_{n-1}}(c_n)}{g_{x_{n-1}}(c_{n-1})}\\
&=\frac{h_{x_0}(c_1)}{g_{x_1}(c_1)}\dots\frac{h_{x_{n-2}}(c_{n-1})}{g_{x_{n-1}}(c_{n-1})}h_{x_{n-1}}(c_n)
\qquad\text{since $\alpha=x_0$}\\
&=f_{x_0}(c_1)\ldots f_{x_{n-2}}(c_{n-1})h_{x_{n-1}}(c_n)\,,
\end{align*}
the last equality by using Lemma~\ref{lem:1}, which applies since $F$
is M\"obius by assumption. 

By the chain rule for the fibred valuation~$F$ on the one hand, and
according to the definition of the graded M\"obius transform
(\S~\ref{sec:grad-mobi-transf}) on the other hand, we obtain thus:
$q_n=h_\alpha(z)$.

Returning to~(\ref{eq:43}), applying the second M\"obius inversion
formula (\S~\ref{sec:second-mobi-invers}) yields:
\begin{align*}
  \pr_\alpha(\up x)&=\sum_{z\in\M(x)}h_\alpha(z)=f_\alpha(x)\,.
\end{align*}
This completes the proof of the theorem.
\end{proof}

\subsubsection{Effectiveness of M\"obius conditions}
\label{sec:effect-mobi-cond}

Since Theorems~\ref{thr:1} and~\ref{thr:2} entirely characterise
Markov measures by means of M\"obius fibred valuations, it is natural
to examine to which extent this condition is effective. 

A natural question is the following: given a family of real numbers
$\lambda=(\lambda_\alpha(a))_{(\alpha,a)\in\XS\times\Sigma}$\,, can we
effectively determine the existence of a Markov measure
$\pr=(\pr_\alpha)_{\alpha\in\XS}$ such that $\pr_\alpha(\up
a)=\lambda_\alpha(a)$ for all $(\alpha,a)\in\XS\times\Sigma$? It
amounts to knowing whether:
\begin{enumerate}\tightlist
\item\label{item:11} The family $\lambda$ defines a fibred valuation.
\item\label{item:12} In this case, the fibred valuation has to be
  M\"obius. 
\end{enumerate}

The first point is solved by Proposition~\ref{prop:6}: the family
$\lambda$ must satisfy the concurrency equations. These are a finite
number of equalities to be satisfied.

For the second point, we observe that the M\"obius conditions stated
in Definition~\ref{def:6} consist in a series of equalities and
another series of inequalities. Assuming that the first point is
fulfilled, the equalities and the inequalities only involve polynomial
expressions of the terms $\lambda_\alpha(a)$, for $(\alpha,a)$ ranging
over~$\XS\times\Sigma$.

Finally, the question is entirely answered by means of polynomial
conditions involving only the terms of the family~$\lambda$. Giving a
parametric form for all M\"obius fibred valuations, and thus all
Markov measures, is yet another question which is not treated here.

Next subsection is devoted to the study of an example illustrating the
use of the machinery described above.

\subsection{Examples of Markov measures on a Rabati tiling}
\label{sec:example-mark-meas}

\subsubsection{Setting}
\label{sec:setting}

We consider the $4$-Rabati strip in line (\S~\ref{sec:rabati-tiling}),
of which we depict in Figure~\ref{fig:rabataifour} the five possible
domino tilings. Recall that the tilings correspond to the set of five
cliques $\C=\{1,a,b,c,a\cdot c\}$ of the trace monoid
\begin{gather*}
  \M=\langle a,b,c\;|\;ac=ca\rangle\,.
\end{gather*}

\begin{figure}[h]
  \centering
  \begin{tabular}{ccc}
    \begin{picture}(80,40)
\multiput(0,0)(20,0){5}{\line(0,1){40}}
\multiput(0,0)(0,40){2}{\line(1,0){80}}
\put(20,20){\circle*{4}$a$}
\put(40,20){\circle*{4}$b$}
\put(60,20){\circle*{4}$c$}
    \end{picture}&
    \begin{picture}(80,40)
\put(0,0){\line(0,1){40}}
\put(0,20){\line(1,0){40}}
\multiput(40,0)(20,0){3}{\line(0,1){40}}
\multiput(0,0)(0,40){2}{\line(1,0){80}}
\put(20,20){\circle*{4}$a$}
\put(60,20){\circle*{4}$c$}
    \end{picture}&
    \begin{picture}(80,40)
\multiput(0,0)(20,0){2}{\line(0,1){40}}
\multiput(60,0)(20,0){2}{\line(0,1){40}}
\put(20,20){\line(1,0){40}}
\multiput(0,0)(0,40){2}{\line(1,0){80}}
\put(40,20){\circle*{4}$b$}
    \end{picture}\\
    \begin{picture}(80,40)
\put(80,0){\line(0,1){40}}
\put(40,20){\line(1,0){40}}
\multiput(0,0)(20,0){3}{\line(0,1){40}}
\multiput(0,0)(0,40){2}{\line(1,0){80}}
\put(20,20){\circle*{4}$a$}
\put(60,20){\circle*{4}$c$}
    \end{picture}&
    \begin{picture}(80,40)
\multiput(0,0)(40,0){3}{\line(0,1){40}}
\multiput(0,0)(0,20){3}{\line(1,0){80}}
\put(20,20){\circle*{4}$a$}
\put(60,20){\circle*{4}$c$}
    \end{picture}
  \end{tabular}
  \caption{\small\textsl{The five tilings of the $4$-Rabati strip
      corresponding, from left to right, to the cliques $1$, $a$, $b$
      on the top line and to the cliques $c$ and $a\cdot c$ on the
      bottom line.}}
\label{fig:rabataifour}
\end{figure}
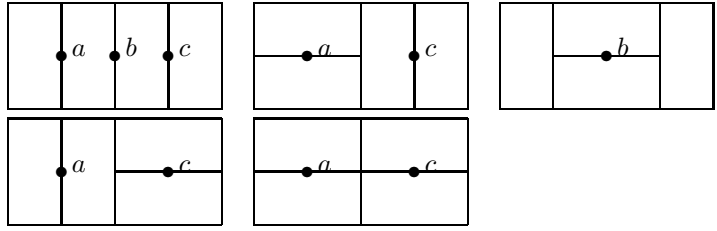

The partial action generated by the enabled flips, which corresponds
to the tip-top action of the trace monoid~$\M$
(\S~\ref{sec:tip-top-action}), is conveniently represented by a graph,
depicted on Figure~\ref{fig:pooqjpqpq}--$(a)$. Since each elementary action is
reversible, the graph is undirected.

\begin{figure}[h]
\centering
$\begin{array}{ccc}
  \xymatrix@C=3.5em{&1\ar@{-}[dl]_{a}\ar@{-}[d]_{b}\ar@{-}[dr]^{c}\\
a\ar@{-}[dr]_{c}&b&c\ar@{-}[dl]^a
\\
&a\cdot c
}
&&
  \xymatrix@C=3.5em{&1\ar@/_/[dl]_{p}\ar@/^/[d]^{q}\ar@/^/[dr]^{p}\\
a\ar@/_/_{p}[dr]\ar_1[ur]&b\ar@/^/^{1}[u]&c\ar@/^/^{p}[dl]\ar^1[ul]
\\
&a\cdot c\ar@/_/_{1}[ul]\ar@/^/^{1}[ur]
}\\
\\[-.5em]
(a)&\strut\hspace{3em}\strut&(b)
\end{array}$
\caption{\small\textsl{$(a)$---Graph of the action of flips on the
    $4$-Rabati tilings. Vertices are labelled by states (cliques) and
    edges are labelled by elementary actions (letters). $(b)$---Graph
    of the probabilistic actions of flips on the same Rabati tiling,
    with probabilistic parameters $q$ and $p=1-\sqrt q$.}}
  \label{fig:pooqjpqpq}
\end{figure}
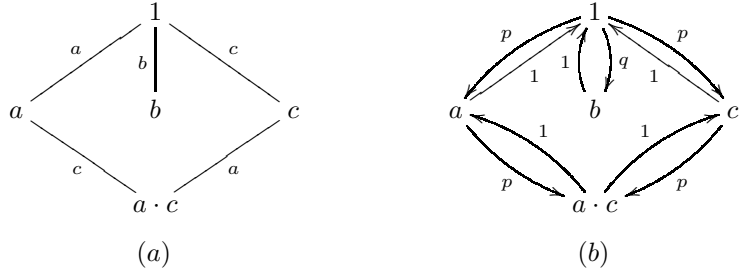

\subsubsection{Markov measures for the partial action}
\label{sec:mark-meas-part}

In order to fit with the previous setting, we have to consider the
 action $(\C\cup\{\bot\})\times\M\to(\C\times\{\bot\})$, as
described in Proposition~\ref{prop:2}, \S~\ref{sec:partial-actions}. Introducing the
probabilistic dynamics however, we impose the additional constraint
that $\pr_\gamma(\up\bot)=0$ for all $\gamma\in\C$, so that $\bot$ is
never reached. In the following, it amounts to merely omitting~$\bot$.

\subsubsection{Markov measures with symmetries}
\label{sec:markov-measures-with}

Let us tackle the task of determining all M\"obius fibred valuations
$F=(f_\gamma)_{\gamma\in\C}$\,, if any, with the following symmetry
properties:
\begin{align}
\label{eq:32}
  f_1(a)&=f_1(c)\,,&f_{a\cdot c}(a)&=f_{a\cdot c}(c)\,,
\end{align}
and such that $f_\gamma(x)>0$ whenever $x$ is a letter enabled
at~$\gamma$.

\subsubsection{Parameters}
\label{sec:parameters}

In a direct approach intending to illustrate the above notions, let us
introduce the following parameters that entirely encode our problem:
\begin{align*}
  p&=f_1(a)=f_1(c)&q&=f_1(b)\\
r&=f_a(c)&r'&=f_a(a)\\
t&=f_c(a)&t'&=f_c(c)\\
u&=f_{a\cdot c}(a)=f_{a\cdot c}(c)&q'&=f_b(b)
\end{align*}

We write down the concurrency equations (Proposition~\ref{prop:6}):
\begin{align*}
&\text{At $1$\,:}&f_1(a)f_a(c)&=f_1(c)f_c(a)&\text{i.e.\ }pr&=pt\\
&\text{At $a$\,:}&f_a(a)f_1(c)&=f_a(c)f_{a\cdot c}(a)&\text{i.e.\
}r'p&=ru\\
&  \text{At $b$\,:}&&\text{none}\\
&\text{At $c$\,:}&f_c(c)f_1(a)&=f_c(a)f_{a\cdot c}(c)&\text{i.e.\ }
t'p&=tu\\
&\text{At $a\cdot c$\,:}&f_{a\cdot c}(a)f_c(c)&=f_{a\cdot
  c}(c)f_a(a)&\text{i.e.\ } ut'&=ur'
\end{align*}

Since $u,p>0$ we deduce at once: $r=t$ and $r'=t'$ and we are thus
left with parameters $p,q,q',r,r',u$ with the only equation:
\begin{align}
\label{eq:31}
  r'p&=ru\,.
\end{align}

We write down the M\"obius equations $h_\gamma(1)=0$ for
$\gamma$ ranging over~$\C$: 
\begin{align}
\label{eq:28}
&\text{At $1$\,:}&1-2p-q+pr&=0\\
\label{eq:29}
&\text{At $a$\,:}&1-r'-r+r'p&=0\\
\label{eq:27}
&  \text{At $b$\,:}&1-q'&=0\\
\notag&\text{At $c$\,:}&\text{same as~(\ref{eq:29})}\\
\label{eq:30}
&\text{At $a\cdot c$\,:}&1-2u+ur'&=0
\end{align}

\subsubsection{Solutions}
\label{sec:solutions}

Equation~(\ref{eq:30}) re-writes as: $u(1-r')=1-u$\,. Replacing $r'p$
by $ru$ in~(\ref{eq:29}), thanks to~(\ref{eq:31}), and then
multiplying by $u$ yields thus:
\begin{gather*}
  (1-u)(1-ru)=0\,.
\end{gather*}

Assume that $u\neq1$. Then $ru=1$, but since both $r$ and $u$ lie in
$(0,1]$ it implies $r=u=1$, and similarly thanks to~(\ref{eq:31}),
$r'=p=1$, but then $q=0$ by~(\ref{eq:28}), contradicting our
assumption  $f_\gamma(x)>0$ for all enabled letters~$x$.  

Therefore $u=1$, and thus $r'=1$ by~(\ref{eq:30}), and $p=r$
by~(\ref{eq:31}).  We now only have two parameters $p$ and $q$ related
by $1-2p-q+p^2=0$, yielding $p=1-\sqrt q$. Here the additional
inequalities from M\"obius conditions are trivially satisfied for
$q\in(0,1)$.  

We obtain thus a continuum of Markov measures satisfying the symmetry
conditions~(\ref{eq:32}). All the parameters of the graded valuation
corresponding to the Markov measure are deduced from~$q$, using the
above relations. They are graphically gathered in
Figure~\ref{fig:pooqjpqpq}--$(b)$. So for instance, starting from the
tiling correspond to the empty clique, the probability of obtaining
the trace $a^2\cdot c^2\cdot b^2\cdot a$ is:
\begin{align}
\notag
  \pr_1(\up a^2\cdot c^2\cdot b^2\cdot
  a)&=f_1(a)f_a(a)f_1(c)f_c(c)f_1(b)f_b(b)f_1(a)\\
\label{eq:48}
&=pr'pr'qq'p=p^3q=(1-\sqrt q)^3q\,.
\end{align}

It could have been computed alternatively, yielding the same result:
\begin{align}
\notag
  \pr_1(\up a^2\cdot c^2\cdot b^2\cdot
  a)&=\pr(\up a\cdot c\cdot c\cdot a\cdot b^2\cdot a)\\
\notag
&=f_1(a)f_a(c)f_{ac}(c)f_a(a)f_1(b)f_b(b)f_1(a)\\
\label{eq:49}
&=prur'qq'p=p^3q=(1-\sqrt q)^3q\,.
\end{align}

In order to construct Markov measures for this example, we have used a
direct approach, writing down the concurrency equations and the
M\"obius equations, and solving them by hand. We shall see below in
next section a more generic construction of Markov measures.





\section{Uniform measures}
\label{sec:uniform-measures}

In~\S~\ref{sec:example-mark-meas}, we have seen examples of Markov
measures $(\pr_\alpha)_{\alpha\in\XS}$ for a trace monoid action
$\XS\times\M\to\XS$, with the following additional property:
\begin{gather}
\label{eq:33}
  \forall\alpha\in\XS\quad\forall x\in\M\quad\pr_\alpha(\up x)>0\iff x\in\M_\alpha\,,
\end{gather}
where $\M_\alpha=\{x\in\M\tq\alpha\cdot x\neq\bot\}$ is the trace
language associated with the partial action
(\S~\ref{sec:trace-lang-assoc}).

Property~(\ref{eq:33}) is natural when dealing with a general concurrent system:
the probabilistic dynamics should only concern enabled actions. The existence of Markov measures satisfying this additional constraint was not guaranteed by the existence Theorem~\ref{thr:2}. It is the aim of this section to prove this
result for all irreducible concurrent systems. The measure that we obtain generalises to concurrent systems the uniform measure for irreducible sub-shifts of finite type, due to Parry~\cite{parry64,lind95,kitchens97}.

\subsection{Growth series and characteristic root}
\label{sec:unif-meas-assoc}

We put aside for a moment the probabilistic dynamics, and focus on the
combinatorics of concurrent systems. In
subsection~\ref{sec:constr-unif-meas}, we will see how both aspects,
probability and combinatorics, combine together and lead us to a
notion of uniform measure.

Throughout this section, we consider a concurrent system $(\M,\XS,\bot)$, with $\M=\M(\Sigma,I)$.

\subsubsection{Growth series}
\label{sec:growth-series-1}

For each $\alpha\in\XS\setminus\{\bot\}$, let $Z_\alpha(t)$ be the power series defined by:
\begin{gather*}
  Z_\alpha(t)=\sum_{x\in\M_\alpha}t^{|x|}\,,
\end{gather*}
where $\M_\alpha$ denotes the trace language associated with the
action. Let $t_\alpha$ be the radius of convergence
of~$Z_\alpha$\,. Since $Z_\alpha$ has only non negative terms,
$t_\alpha$~is a singularity of~$Z_\alpha$\,.

For each integer $k\geq0$, let $\M_\alpha(k)$ denote the subset of
traces $\M_\alpha(k)=\{x\in\M_\alpha\tq|x|=k\}$\,. Since at least one
letter is enabled at each state, it is easy to see that
$\#\M_\alpha(k)\geq1$. And since $Z_\alpha(t)$ also writes as
$Z_\alpha(t)=\sum_{k\geq0}\#\M_\alpha(k)t^k$\,, we have $t_\alpha\leq1$.

\subsubsection{Characteristic root of an irreducible action}
\label{sec:char-root-an}

We show below that the real~$t_\alpha$\,, which characterises the
growth of the monoid acting on~$\XS$ and starting from~$\alpha$, does
not depend on $\alpha$ if the action is irreducible.

\begin{lemma}
  \label{lem:3}
  If $\alpha$ leads to~$\beta$ ($\alpha\leads\beta$ with the notation
  of~{\normalfont\S~\ref{sec:essential-states}}), then $t_\alpha\leq t_\beta$\,.
\end{lemma}

\begin{proof}
  Recall that we are only dealing with power series with non negative
  coefficients. Assuming that $\alpha\leads\beta$ holds, let
  $x\in\M_\alpha$ be such that $\beta=\alpha\cdot x$. Then we have
  $x\cdot y\in\M_\alpha$ for every $y\in\M_\beta$\,, and therefore:
  \begin{align*}
    Z_\alpha(t)&\geq\sum_{y\in\M_\beta}t^{|x\cdot y|} \geq
    t^{|x|}Z_\beta(t)\,.
  \end{align*}
It follows that $Z_\alpha(t)<\infty\implies Z_\beta(t)<\infty$\,,
whence the inequality $t_\beta\geq t_\alpha$\,.
\end{proof}

\begin{corollary}
  \label{cor:2}
  If the concurrent system $(\M,\XS,\bot)$ is irreducible, then $t_\alpha$ is independent
  of~$\alpha$.
\end{corollary}

\begin{proof}
  Any two states $\alpha$ and $\beta$ of an irreducible partial action
  communicate: $\alpha\leads\beta$ and $\beta\leads\alpha$. Hence
  $t_\alpha=t_\beta$ by Lemma~\ref{lem:3}.
\end{proof}

The corollary justifies the following definition.

\begin{definition}
  \label{def:11}
  For an irreducible concurrent system $(\M,\XS,\bot)$, the common value~$t_\alpha$\,,
  for $\alpha$ ranging over~$\XS\setminus\{\bot\}$\,, is called the \emph{characteristic root}
  of the concurrent system. We denote it by~$t_0$\,.
\end{definition}

We shall now study an inversion formula, which extends to concurrent systems the first M\"obius inversion formula valid for trace monoids (\S~\ref{sec:first-mobi-invers}). It will justify the term of
``root'' for~$t_0$\,.





\subsection{An inversion formula for actions of trace monoids}
\label{sec:an-inversion-formula}

\subsubsection{Formal fibred series over non commutative variables}
\label{sec:formal-fibred-series}

We extend the notion of formal series over partially commutative
variables \cite{schuetzenberger61,cartier69} in order to take into
account not only the combinatorics of a trace monoid, but also the
combinatorics of a trace monoid acting on a finite set.

\begin{definition}
  \label{def:7}
  Given a concurrent system $(\M,\XS,\bot)$, we consider for each
  pair $(\alpha,\beta)$ of states, the trace language
  $\M_{\alpha,\beta}$ of those traces leading from $\alpha$
  to~$\beta$, hence defined by:
  \begin{gather}
    \label{eq:34}
\M_{\alpha,\beta}=\{x\in\M\tq x\in\M_\alpha\wedge \alpha\cdot x=\beta\}\,.
  \end{gather}

  A \emph{fibred formal series} is a square matrix
  $F=(F_{\alpha,\beta})_{(\alpha,\beta)\in\XS\times\XS}$\,, where each
  $F_{\alpha,\beta}$ is a mapping $F_{\alpha,\beta}:\M\to\bbZ$\,, such
  that holds:
  \begin{gather*}
    \forall x\in\M\quad F_{\alpha,\beta}(x)\neq0\implies x\in\M_{\alpha,\beta}\,.
  \end{gather*}
  We see $F_{\alpha,\beta}$ as a formal sum indexed
    by~$\M_{\alpha,\beta}$\,. We denote by $\bbZ\cro\M\XS$ the set of
    fibred formal series.

An element $F\in\bbZ\cro\M\XS$ is a \emph{formal fibred polynomial} whenever
$F_{\alpha,\beta}$ has a finite support, for any
    $(\alpha,\beta)\in\XS\times\XS$\,. 
\end{definition}

\subsubsection{$\bbZ$-algebra structure for $\bbZ\cro\M\XS$}
\label{sec:bbz-algebra-struct}

If $F_{\alpha,\beta}$ and $F'_{\alpha,\beta}$ are two formal sums
indexed by~$\M_{\alpha,\beta}$\,, we define their sum
$F_{\alpha,\beta}+F'_{\alpha,\beta}$ term by term, leading to another
formal sum indexed by~$\M_{\alpha,\beta}$\,. The scalar multiplication
$\lambda F_{\alpha,\beta}$ is defined by the scalar multiplication of
each term.

If $F_{\alpha,\beta}$ and $G_{\beta,\gamma}$ are two formal sums
indexed by $\M_{\alpha,\beta}$ and $\M_{\beta,\gamma}$ respectively,
we can define their Cauchy product $F_{\alpha,\beta}G_{\beta,\gamma}$
as the a formal sum indexed by~$\M_{\alpha,\gamma}$\,, each term being
defined by the following finite sum:
\begin{gather}
\label{eq:35}
\forall x\in\M_{\alpha,\gamma}\quad
(F_{\alpha,\beta}G_{\beta,\gamma})(x)=
\sum_{\substack{(y,z)\in\M_{\alpha,\beta}\times\M_{\beta,\gamma}\;:\\
    y\cdot z=x}} F_{\alpha,\beta}(y)G_{\beta,\gamma}(z)
\end{gather}

Next, we define the product of two fibred formal series $F$ and $G$ by
the matrix product:
\begin{gather*}
  (FG)_{\alpha,\gamma}=\sum_{\beta\in\XS}F_{\alpha,\beta}G_{\beta,\gamma}\,,
\end{gather*}
where each product in the sum is taken as in~(\ref{eq:35}).

Then clearly, $\bbZ\cro\M\XS$ has the structure of a
$\bbZ$-algebra, of which fibred formal polynomials are a
sub-algebra. The identity element is the identity matrix
$I=(I_{\alpha,\beta})_{(\alpha,\beta)\in\XS\times\XS}$\,, which terms
are defined by:
\begin{gather*}
  \forall(\alpha,\beta)\in\XS\times\XS\quad\forall
  x\in\M_{\alpha,\beta}\quad
  I_{\alpha,\beta}(x)=\un_{\{\alpha=\beta\}}\un_{\{x=1\}}\,.
\end{gather*}

\subsubsection{Zeta  series and inversion formula}
\label{sec:zeta-mobius-fibred}

We shall be interested by two elements of $\bbZ\cro\M\XS$ in particular.

\begin{definition}
  \label{def:8}
The zeta fibred formal series associated to the concurrent system $(\M,\XS,\bot)$ is the element $\zeta\in\bbZ\cro\M\XS$
defined by:
\begin{gather*}
  \forall(\alpha,\beta)\in\XS\times\XS\quad\forall
  x\in\M_{\alpha,\beta}\quad \zeta_{\alpha,\beta}(x)=1\,.
\end{gather*}

For each pair $(\alpha,\beta)\in\XS\times\XS$\,, let
$\C_{\alpha,\beta}$ denote the set of cliques leading from $\alpha$
to~$\beta$, if any: $\C_{\alpha,\beta}=\C\cap\M_{\alpha,\beta}$\,. The
M\"obius fibred polynomial of the concurrent system $(\M,\XS,\bot)$ is the element $\mu\in\bbZ\cro\M\XS$ defined by:
\begin{gather*}
  \forall(\alpha,\beta)\in\XS\times\XS\quad\forall
  x\in\M_\alpha\quad\mu_{\alpha,\beta}(x)=(-1)^{|x|}\un_{\{x\in\C_{\alpha,\beta}\}}\,.
\end{gather*}
\end{definition}

Then we have the following inversion formula.

\begin{theorem}
  \label{thr:3}
  For any concurrent system $(\M,\XS,\bot)$, the M\"obius fibred polynomial is the formal inverse of the zeta
  fibred formal series:
\begin{gather*}
  \mu\zeta=\zeta\mu=I.
\end{gather*}
\end{theorem}

\begin{proof}
By definition of the product in $\bbZ\cro\M\XS$, we have for
$(\alpha,\beta)\in\XS\times\XS$ and for $x\in\M_{\alpha,\beta}$:
\begin{align*}
  (\mu\zeta)_{\alpha,\beta}(x)&=\sum_{\gamma\in\XS}\Bigl(
\sum_{y\in\C_{\alpha,\gamma}\tq y\leq x}(-1)^{|y|}
\Bigr)\\
&=\sum_{y\in\C_\alpha\tq y\leq x}(-1)^{|y|}\,,
\end{align*}
where $\C_\alpha$ denotes $\C_\alpha=\C\cap\M_\alpha$\,. The binomial
formula yields thus:
\begin{align*}
  (\mu\zeta)_{\alpha,\beta}(x)&=
  \begin{cases}
    1,&\text{if $x=1$},\\
0,&\text{otherwise.}
  \end{cases}
\end{align*}
But $x=1$ implies $\alpha=\beta$, and so we have proved
$\mu\zeta=I$. The equality $\zeta\mu=I$ is proved similarly.
\end{proof}

\subsubsection{A morphism toward a single formal variable}
\label{sec:morph-toward-single}

In order to exploit the above inversion formula (Theorem~\ref{thr:3})
for counting purposes, we follow a classical pattern by projecting
fibred formal series onto series over a single formal variable.

\begin{definition}
  \label{def:9}
  Given a set $\XS$, the algebra of single fibred formal series is the algebra
  $\bbZ_{\XS}[[t]]$ of square matrices indiced by~$\XS$, where each entry is
  a formal series over the single formal variable~$t$. 

  The algebra $\bbZ_\XS[t]$ of single formal polynomials is the
  sub-algebra of\/ $\bbZ_\XS[[t]]$ such that each entry is a polynomial
  in~$t$.
\end{definition}

It must be noted that if $F$ and $G$ are two formal series over~$t$
given under the form:
\begin{align*}
  F&=\sum_{x\in\M}F_xt^{|x|}\,,&G&=\sum_{x\in\M}G_xt^{|x|}\,,
\end{align*}
then their product is given by the following Cauchy product formula
with respect to the multiplication in the monoid~$\M$:
\begin{gather}
\label{eq:44}
  FG=\sum_{x\in\M}\Bigl(\sum_{(y,z)\in\M\times\M\tq y\cdot z=x}F_yG_z\Bigr)t^{|x|}\,.
\end{gather}

\begin{proposition}
  \label{prop:7}
The mapping\/ $\Phi:\bbZ\cro\M\XS\to\bbZ_\XS[[t]]$ defined as follows,
for $F\in\bbZ\cro\M\XS$:
\begin{gather*}
  \forall(\alpha,\beta)\in\XS\times\XS\quad
(\Phi F)_{\alpha,\beta}=\sum_{x\in\M}F_{\alpha,\beta}(x)t^{|x|}
\end{gather*}
is a morphism of\/ $\bbZ$-algebra.
\end{proposition}

\begin{proof}
  The property $\Phi I=I$ is obvious. The property
  $\Phi(FG)=\Phi(F)\Phi(G)$ is the matter of a simple verification
  left to the reader based on~(\ref{eq:44}).
\end{proof}

\subsubsection{The theta polynomial}
\label{sec:single-fibred-series}

We now introduce the $Z$ series, which strong combinatorial meaning
justifies the introduction of the above material.

\begin{definition}
  \label{def:10}
  Let $(\M,\XS,\bot)$ be a concurrent system. 
The associated $Z$ series is the single fibred formal series defined by:
\begin{gather*}
  \forall(\alpha,\beta)\in\XS\times\XS\quad Z_{\alpha,\beta}(t)=\sum_{x\in\M_{\alpha,\beta}}t^{|x|}.
\end{gather*}

The \emph{M\"obius single polynomial} $\mu\in\bbZ_\XS[t]$ is defined by:
\begin{gather*}
  \forall(\alpha,\beta)\in\XS\times\XS\quad
  \mu_{\alpha,\beta}(t)=\sum_{\gamma\in\C_{\alpha,\beta}}(-1)^{|\gamma|}t^{|\gamma|}\,.
\end{gather*}

The \emph{theta polynomial} is the standard polynomial $\theta(t)\in\bbZ[t]$
defined as the following determinant:
\begin{gather*}
  \theta(t)=\det\mu(t)\,.
\end{gather*}
\end{definition}

For instance, if the partial action is that of a sub-shift of finite
type generated by an incidence matrix~$M$ over the set of states, then
the single M\"obius polynomial is $\mu(t)=I-tM$. See computations of
$\mu(t)$ and of $\theta(t)$ for actions with concurrency
in~\S~\ref{sec:using-theta-polyn}.

\begin{theorem}
  \label{thr:4}
  Given any concurrent system $(\M,\XS,\bot)$, the M\"obius single polynomial is the formal inverse in\/
  $\bbZ_\XS[[t]]$ of the $Z$ series:
\begin{gather*}
  \mu(t)Z(t)=Z(t)\mu(t)=I\,,
\end{gather*}
where $I$ is the identity matrix defined by $I_{\alpha,\beta}(t)=\un_{\{\alpha=\beta\}}$\,.
\end{theorem}

\begin{proof}
  This derives by applying the morphism $\Phi$ of
  Proposition~\ref{prop:7} to the inversion formula stated in
  Theorem~\ref{thr:3}, since $\Phi(\zeta)=Z(t)$ and
  $\Phi(\mu)=\mu(t)$\,.
\end{proof}

\begin{theorem}
  \label{thr:6}
  Let $(\M,\XS,\bot)$ be an irreducible concurrent system. Then the characteristic root $t_0$ of the concurrent system is the smallest positive root of the theta polynomial.
\end{theorem}

\begin{proof}
  Let $t\in(0,t_0)$\,. Then all terms in $Z(t)$ are non negative
  convergent series, therefore the equality $\mu(t)Z(t)=I$ holds in
  the space of real matrices $M_{n}(\bbR)$, with $n=|\XS|$. Therefore
  $\theta(t)\neq0$.

  Assume, seeking a contradiction, that $\theta(t_0)\neq0$. Then
  $\mu(t)$ is invertible for all $t$ in the closed interval~$[0,t_0]$.
  Therefore the real matrices $Z(t)$ are uniformly bounded
  on~$[0,t_0)$. 

  Observe now that, for $t<t_0$ and for $\alpha\in\XS$, the line
  indiced by $\alpha$ of the real matrix $Z(t)$ has its terms that sum
  up to the convergent growth series~$Z_\alpha(t)$. Hence we obtain
  that $Z_\alpha(t)$ is bounded on $[0,t_0)$, a contradiction which
  completes the proof of the theorem.
\end{proof}

\subsubsection{Reduction to the standard M\"obius inversion formula}
\label{sec:reduct-stand-mobi}

Theorem~\ref{thr:4} entails the first M\"obius inversion formula
(\S~\ref{sec:first-mobi-invers}) as a particular case. Consider indeed
the trivial action of a trace monoid $\M$ over a
singleton~$\{*\}$. Then $Z(t)$ identifies with the growth series
$H(t)=\sum_{x\in\M}t^{|x|}$\,, the M\"obius fibred polynomial is
simply the M\"obius polynomial
$\mu_\M(t)=\sum_{c\in\C}(-1)^{|c|}t^{|c|}$ and the formula in
Theorem~\ref{thr:3} is the inversion formula~(\ref{eq:2}).

\subsection{Construction of the uniform measure}
\label{sec:constr-unif-meas}

\subsubsection{Measures on $\Mbar$}
\label{sec:measures-mbar}

We recall that $\Mbar=\M\cup\BM$ denotes the set of generalised
traces build upon the trace monoid~$\M$. It is equipped with a metrisable topology, for which it is a
compact space (\S~\ref{sec:compactness}). 

Assume given an irreducible concurrent system $(\M,\XS,\bot)$. For $\alpha\in\XS$ and for $t\in(0,t_0)$, let $\nu_{\alpha,t}$ denote the probability measure on~$\Mbar$ defined by:
\begin{align*}
  \nu_{\alpha,t}&=\frac1{Z_\alpha(t)}\sum_{x\in\M_\alpha}t^{|x|}\delta_{\{x\}}\,,&
Z_\alpha(t)&=\sum_{x\in\M_\alpha}t^{|x|}\,.
\end{align*}

\begin{lemma}
  \label{lem:5}
  For each state $\alpha\in\XS$ and for each trace $x\in\M_\alpha$\,,
  the full elementary cylinder $\Up x$ defined by:
\begin{gather*}
  \Up x=\{\xi\in\Mbar\tq x\leq\xi\}
\end{gather*}
is given the $\nu_{\alpha,t}$-probability:
\begin{gather*}
  \nu_{\alpha,t}(\Up x)=t^{|x|}\frac{Z_{\alpha\cdot x}(t)}{Z_\alpha(t)}\,.
\end{gather*}
\end{lemma}

\begin{proof}
  Since $\nu_{\alpha,t}$ only charges~$\M$, $\nu_{\alpha,t}(\Up x)$ is
given by the following sum:
\begin{align*}
  \nu_{\alpha,t}(\Up x)&=\frac1{Z_\alpha(t)}\sum_{y\in\M_{\alpha}\tq
    x\leq y}t^{| y|}\,.
\end{align*}
The traces $y\in\M_\alpha$ such that $x\leq y$ are of the form
$y=x\cdot z$ with $z\in\M_{\alpha\cdot x}$\,, hence:
\begin{align*}
  \nu_{\alpha,t}(\Up x)&=\frac1{Z_\alpha(t)}\sum_{z\in\M_{\alpha\cdot
      x}}t^{|x\cdot z|}=t^{|x|}\frac{Z_{\alpha\cdot x}(t)}{Z_\alpha(t)}\,.
\end{align*}
The proof is complete.
\end{proof}

\begin{lemma}
  \label{lem:6}
  For an irreducible concurrent system $(\M,\XS,\bot)$ of characteristic root~$t_0$\,, the probability measures
  $(\nu_{\alpha,t})_{t<t_0}$ have a weak limit when $t\to
  t_0^-$\,. This limit, denoted by~$\nu_\alpha$\,, only
  charges~$\BM$\,, and satisfies $\nu_\alpha(\up x)>0$ for every
  $x\in\M_\alpha$\,.
\end{lemma}

\begin{proof}
  First, we fix $\alpha\in\XS$.  Since $\Mbar$ is compact, there
  exists a probability measure $\nu_\alpha$ on $\Mbar$ which is a weak
  limit of a sequence $(\nu_{\alpha,t_n})_{n\geq0}$\,, for $t_n\to
  t_0^-$\,.

  For $x\in\M$, the singleton $\{ x\}$ is both open and closed
  in~$\Mbar$. It has thus an empty topological boundary. Therefore, by
  the Portemanteau theorem~\cite{billingsley95}, the following
  equality hods:
\begin{gather*}
  \nu_\alpha(\{x\})=\lim_{n\to\infty}\nu_{\alpha,t_n}(\{x\})=\lim_{n\to\infty}\frac {t_n^{|x|}}{Z_\alpha(t_n)}=0
\end{gather*}
since $t_n\leq t_0\leq 1$ on the one hand, and since $\lim_{t\to
  t_0}Z_\alpha(t)=\infty$ on the other hand. The set $\M$ is
countable, and thus
$\nu_\alpha(\M)=0$. 

The elementary cylinder $\Up x$ is also both open and closed
in~$\Mbar$. Using again the Portemanteau theorem, we have thus:
\begin{gather*}
  \nu_\alpha(\Up x)=\lim_{n\to\infty}\nu_{\alpha,t_n}(\Up x)\,.
\end{gather*}
Since $\nu_\alpha(\M)=0$, it follows:
\begin{gather*}
  \nu_\alpha(\up x)=\lim_{n\to\infty}\nu_{\alpha,t_n}(\Up
  x)=t_0^{|x|}\lim_{n\to\infty}\frac{Z_{\alpha\cdot x}(t_n)}{Z_\alpha(t_n)}\,,
\end{gather*}
the last equality by virtue of Lemma~\ref{lem:5}.

Focus now on the ratio $Z_{\alpha\cdot x}(t)/Z_\alpha(t)$\,. It is a
ratio of power series with non negative coefficients. It has therefore
a limit, for $t\to t_0^-$\,, which is either zero, or a positive real,
or~$+\infty$. We have just seen that, for some sequence $t_n\to
t_0^-$\,, the ratios $Z_{\alpha\cdot x}(t_n)/Z_\alpha(t_n)$ have a non
negative limit. We deduce:
\begin{gather}
\label{eq:38}
  \lim_{t\to t_0^-}\frac{Z_{\alpha\cdot x}(t)}{Z_\alpha(t)}\in[0,\infty)
\end{gather}

It also follows that the probability measure $\nu_\alpha$ has its
values on elementary cylinders $\up x$ entirely determined by these
limits, independently of the sequence~$(t_n)_n$\,. Therefore the
family $(\nu_{\alpha,t})_{t<t_0}$ has $\nu_\alpha$ as weak limit when
$t\to t_0^-$\,.

It remains only to show $\nu_\alpha(\up x)>0$ for
$x\in\M_\alpha$\,. It amounts to showing that the limit
in~(\ref{eq:38}) is positive. Let $\beta=\alpha\cdot x$. Since the
action is irreducible, there exists $y\in\M_\beta$ such that
$\alpha=\beta\cdot y$. Applying~(\ref{eq:38}) to $\beta$ and $y$ in
place of $\alpha$ and~$x$, we have:
\begin{gather*}
  \lim_{t\to t_0^-}\frac{Z_{\beta\cdot y}(t)}{Z_\beta(t)}<\infty
\end{gather*}
which also writes as:
\begin{gather}
\label{eq:42}
  \lim_{t\to t_0^-}\frac{Z_{\alpha\cdot x}(t)}{Z_\alpha(t)}>0\,.
\end{gather}
This completes the proof of the lemma.
\end{proof}

\subsubsection{The Parry cocycle}
\label{sec:parry-cocycle-1}

The above lemma has the following consequence.

\begin{corollary}
  \label{cor:3}
For each pair of states $(\alpha,\beta)\in\XS\times\XS$ of an
irreducible concurrent system $(\M,\XS,\bot)$, the following limit:
\begin{gather}
\label{eq:37}
  \Gamma(\alpha,\beta)=\lim_{t\to t_0^-}\frac{Z_\beta(t)}{Z_\alpha(t)}
\end{gather}
exists and lies in $(0,\infty)$.
\end{corollary}

\begin{proof}
Since the concurrent system  is irreducible, every $\beta\in\XS$ writes as
$\beta=\alpha\cdot x$ for some $x\in\M_\alpha$\,. The statement of the
corollary was then shown in the course of the proof of
Lemma~\ref{lem:6}, in~(\ref{eq:38}) and in~(\ref{eq:42}).
\end{proof}

\begin{definition}
  \label{def:13}
For an irreducible concurrent system $(\M,\XS,\bot)$, the Parry
cocyle is the function $\Gamma(\cdot,\cdot):\XS\times\XS\to(0,\infty)$
defined by\/~\eqref{eq:37}.
\end{definition}

The term ``cocyle'' is appropriate, since the definition~(\ref{eq:37}) of
$\Gamma(\cdot,\cdot)$ makes obvious the two following properties:
\begin{align}
\label{eq:39}
  \forall\alpha\in\XS\quad\Gamma(\alpha,\alpha)&=1\\
\label{eq:40}
\forall(\alpha,\beta,\gamma)\in\XS\times\XS\times\XS\quad
\Gamma(\alpha,\gamma)&=\Gamma(\alpha,\beta)\Gamma(\beta,\gamma)
\end{align}

\subsubsection{Uniform measure}
\label{sec:uniform-measure-1}

We can now combine the results obtained so far to obtain the following
theorem.

\begin{theorem}
  \label{thr:5}
Let $(\M,\XS,\bot)$ be an irreducible concurrent system, with characteristic root $t_0$ and Parry cocycle $\Gamma(\,\cdot\,,\,\cdot\,)$.  Then there
exists a Markov measure, denoted $\nu=(\nu_\alpha)_{\alpha\in\XS}$ and
called the uniform measure, satisfying the following property:
\begin{gather}
\label{eq:41}
  \forall\alpha\in\XS\quad\forall x\in\M_\alpha\quad\nu_\alpha(\up x)=
t_0^{|x|}\Gamma(x,\alpha\cdot x)\,,
\end{gather}
and $\nu_\alpha(\up x)=0$ if $x\notin\M_\alpha$\,.  In particular, the
support of $\nu$ coincides with the support of the partial action.
\end{theorem}

\begin{proof}
  According to Lemma~\ref{lem:6}, we shall define~$\nu_\alpha$\,, for
  each $\alpha\in\XS$, as the weak limit of $\nu_{\alpha,t}$ for $t\to
  t_0^-$\,, and then put $\nu=(\nu_\alpha)_{\alpha\in \XS}$ \,.

  The relation~(\ref{eq:41}) follows from an application of the
  Portemanteau theorem as in the proof of Lemma~\ref{lem:6}, starting
  from the result of Lemma~\ref{lem:5} and taking the limit $t\to
  t_0^-$\,. Since the Parry cocyle is positive, $\nu_\alpha(\up
  x)>0$~for all $x\in\M_\alpha$\,.

  If $x\notin\M_\alpha$\,, then $\nu_{\alpha,t}(\Up x)=0$ for all
  $t<t_0$ by the very definition of~$\nu_{\alpha,t}$ and since
  $\M_\alpha$ is downward closed according to
  Proposition~\ref{prop:3}. Therefore, passing to the limit still
  legitimated by the Portemanteau theorem, we obtain $\nu_\alpha(\up
  x)=0$. Hence $\nu=(\nu_\alpha)_{\alpha\in\XS}$ has same support as
  the action.

It remains to show that $\nu$ thus defined is Markovian. For this, we verify
 the validity of the defining relation:
 \begin{gather}
   \forall\alpha\in\XS\quad\forall x,y\in\M\quad\nu_\alpha(\up(x\cdot
   y))=\nu_\alpha(\up x)\nu_{\alpha\cdot x}(\up y)\,.
 \end{gather}
 If $x\notin\M_\alpha$ or if $y\notin\M_{\alpha\cdot x}$\,, then both
 members of the above equality are zero. And for $x\in\M_\alpha$ and
 $y\in\M_{\alpha\cdot x}$\,, using the cocyle property~(\ref{eq:40}),
 we have:
\begin{align*}
  \nu_\alpha(\up (x\cdot y))&=t_0^{|x\cdot y|}\Gamma(\alpha,\alpha\cdot
  x\cdot y)\\
&=t_0^{|x|}\Gamma(\alpha,\alpha\cdot x) t_0^{|y|}\Gamma(\alpha\cdot
x,(\alpha\cdot x)\cdot y)\\
&=\nu_\alpha(\up x)\nu_{\alpha\cdot x}(\up y)
\end{align*}
The proof is complete.
\end{proof}

\subsection{Example: uniform measure of a Rabati tiling}
\label{sec:unif-meas-rabati}

We illustrate the construction of the uniform measure on the simple
example already encountered in~\S~\ref{sec:example-mark-meas} of the
$4$-Rabati tiling in line.

\subsubsection{Continuing the direct approach}
\label{sec:cont-brute-force}

Let us tackle the problem of determining the uniform measure directly
through its M\"obius fibred valuation $F=(f_\gamma)_{\gamma\in\C}$\,.

For symmetry reasons, it is clear that $F$ satisfies the conditions
$f_1(a)=f_1(c)$ and $f_{a\cdot c}(a)=f_{a\cdot c}(c)$, hence it fits
within the framework previously studied
in~\S~\ref{sec:example-mark-meas}.

At the end of~\S~\ref{sec:solutions}, we were left with the only
parameter $q=f_1(b)$, from which all other values characterising the
M\"obius valuation $F$ can be deduced. Given the special form of the
uniform measure, we have on the one hand:
\begin{align*}
  f_1(b)&=t_0\Gamma(1,b)\\
f_b(b)&=t_0\Gamma(b,1)
\end{align*}
and on the other hand, by the cocyle property~(\ref{eq:40}) and since $f_b(b)=1$:
\begin{align*}
  \Gamma(b,1)&=t_0^{-1}&\Gamma(1,b)&=\Gamma(b,1)^{-1}
\end{align*}
Therefore:
\begin{align*}
  q&=t_0\Gamma(1,b)=t_0^2
\end{align*}

Applying the same reasoning for $p=f_1(a)$ we have, using that $f_a(a)=r'=1$:
\begin{align*}
  p&=t_0\Gamma(1,a)&1&=t_0\Gamma(a,1)=t_0\Gamma(1,a)^{-1}=\frac{t_0^2}{p}
\end{align*}
Therefore $p=t_0^2=q$\,. But we also found $1-2p-q+p^2=0$, therefore:
\begin{align*}
  1-3p+p^2=0
\end{align*}
and we recognise that $p=\frac{3-\sqrt5}2$ is the only root in $(0,1)$ of the M\"obius
polynomial of the monoid $\M=\langle a,b,c\;|\; ac=ca\rangle$. The
characteristic root of the action is:
\begin{gather*}
  t_0=\sqrt{\frac{3-\sqrt5}{2}}\,.
\end{gather*}

Since we have obtained the value of~$t_0$ and we can determine the
M\"obius fibred valuation, we can deduce the values of the Parry
cocyle, through the formula
$f_\gamma(x)=t_0^{|x|}\Gamma(\gamma,\varphi(\gamma,x))$. By the cocyle
property, it is enough to give the values of $\Gamma(1,\cdot)$:
\begin{align*}
\Gamma(1,a)&=t_0&\Gamma(1,b)&=t_0&\Gamma(1,c)&=t_0&\Gamma(1,a\cdot c)&=t_0^2
\end{align*}
We obtain thus the following formula for the cocycle:
\begin{gather}
\label{eq:46}
  \forall\gamma,\gamma'\in\C\quad\Gamma(\gamma,\gamma')=t_0^{|\gamma'|-|\gamma|}\,.
\end{gather}

We can check that we obtain the same value for $\pr_1(\up x)$\,, with
$x=a^2\cdot c^2\cdot b^2\cdot a$\,, as we found in~(\ref{eq:48}) and in~(\ref{eq:49}):
\begin{gather*}
  \pr_1(\up x)=t_0^{|x|}\Gamma(1,\varphi(1,x))=t_0^7\Gamma(1,a)=t_0^8\,,
\end{gather*}
which corresponds to the value $p^3q$ found previously, since $p=q=t_0^2$\,.

\subsubsection{Using the theta polynomial}
\label{sec:using-theta-polyn}

We could also have obtained the value of $t_0$ by applying the result
of Theorem~\ref{thr:6}, which implies to determine the theta
polynomial of the action. The M\"obius single fibred polynomial
(Definition~\ref{def:10} in~\S~\ref{sec:single-fibred-series}) is the
following matrix indexed by~$\C$:
\begin{align*}
  \mu(t)=
  \begin{array}{c}
    1\\a\\b\\c\\a\cdot c
  \end{array}
  \begin{pmatrix}
    1&-t&-t&-t&t^2\\
    -t&1&0&t^2&-t\\
    -t&0&1&0&0\\
    -t&t^2&0&1&-t\\
    t^2&-t&0&-t&1
  \end{pmatrix}
\end{align*}
The theta polynomial is the determinant of the above matrix, found to
be equal to:
\begin{gather*}
  \theta(t)=(1-3t^2+t^4)(1-t^2)^2
\end{gather*}

\emph{Via} the result of Theorem~\ref{thr:6}, we recover that $t_0^2$
is the root in $(0,1)$ of the M\"obius polynomial $1-3p+p^2$ of the
monoid. With this method however, the apparition of the M\"obius
polynomial seems miraculous. But we shall obtain yet another time this
result by studying more generally tip-top actions, of which Rabati
tiling actions are a particular case
(\S~\ref{sec:tip-top-action}--\ref{sec:rabati-tiling}). The occurrence
of the M\"obius polynomial will then be given a more natural
explanation.

\subsection{Example: uniform measure of the tip-top action}
\label{sec:uniform-measure-tip}

In this subsection, we study the uniform measure associated with the
tip-top action of a trace monoid over its set of cliques
(\S~\ref{sec:tip-top-action}). 

We recall from~\S~\ref{sec:rabati-tiling} that the action of flips on
a Rabati tiling is a particular case of a tip-top action. Hence our
study will allow us to re-obtain results on the uniform measure of the
$4$-Rabati tiling in line, which we previously obtained by hand
through M\"obius equations and symmetry considerations
in~\S~\ref{sec:example-mark-meas} and~\S~\ref{sec:unif-meas-rabati}.

Our first result, below, shows that the presence of $1$'s on all
ascending arrows in the graph of Figure~\ref{fig:pooqjpqpq}--$(b)$ is
due to a general fact.

\begin{lemma}
\label{lem:7}
  Let\/ $\pr=(\pr_\gamma)_{\gamma\in\C}$ be a Markov measure on the
  tip-top action $\C\times\M\to\C$, and let
  $(f_\gamma)_{\gamma\in\C}$ be the associated M\"obius
  valuation. Then holds:
  \begin{gather*}
    \forall\gamma\in\C\quad\forall a\in\Sigma\quad a\leq
    \gamma\implies f_\gamma(a)\in\{0,1\}\,.
  \end{gather*}
\end{lemma}

\begin{proof}
  Let $\gamma\in\C$ and let $a\in\Sigma$ be such that
  $a\leq\gamma$. Then by induction, it is easy to see that for any
  trace $x\in\M_\gamma$ of the form $x=a_1\cdot\ldots\cdot a_n$ with
  $a_i\neq a$ for all integers $i\in\{1,\ldots,n\}$, then
  $a_i\parallel a$ holds, and $a\in\M_{\gamma\cdot x}$ also holds. In
  words: as long as $a$ has not been played, it cannot be disabled.

  Therefore, assuming that $f_\gamma(a)>0$ holds, then the Markov
  chain of states-and-cliques has positive probability to hit a
  state-and-clique where the clique involves~$a$, and the set of such
  states-and-cliques will remain accessible as long as no clique
  involving $a$ has been produced. Hence this will eventually occur
  with probability~$1$.

  Returning to the valuation point of view, it implies that the
  cylinder $\up a$ has $\pr_\gamma$-probability~$1$. Hence
  $f_\gamma(a)=1$, which was to be proved.
\end{proof}

\begin{corollary}
  \label{cor:4}
  Let $(f_\gamma)_{\gamma\in\C}$ be the M\"obius fibred valuation
  associated with the uniform measure of the tip-top action. Then
  holds:
\begin{enumerate}\tightlist
\item\label{item:13} $\forall\gamma\in\C\quad\forall a\in\Sigma\quad
  a\leq \gamma\implies f_\gamma(a)=1$\,.
\item\label{item:14} $\forall\gamma\in\C\quad f_\gamma(\gamma)=1$.
\end{enumerate}
\end{corollary}

\begin{proof}
  Let $\gamma\in\C$ and let $a\in\Sigma$ be such that $a\leq\gamma$.
  By Theorem~\ref{thr:5}, the uniform measure has same support as the
  support of the action. Therefore, since $a\in\M_\gamma$ if
  $a\leq\gamma$, one has $f_\gamma(a)>0$ and thus $f_\gamma(a)=1$ by
  Lemma~\ref{lem:7}, which proves point~\ref{item:13}.

For point~\ref{item:14}, let $\gamma\in\C$. Then $\gamma$ writes as
$\gamma=a_1\cdot\ldots\cdot a_i$\,, where the $a_i$ are pairwise
parallel. By the chain rule for the fibred valuation on the one hand,
and using the result just proved on the other hand, we have:
\begin{gather*}
  f_\gamma(\gamma)=f_\gamma(a_1)f_{a_2\cdot\ldots\cdot
    a_{i}}(a_2\cdot\ldots\cdot a_i)=f_{a_2\cdot\ldots\cdot
    a_{i}}(a_2\cdot\ldots\cdot a_i).
\end{gather*}
Hence, by induction, $f_\gamma(\gamma)=1$, which was to be proved.
\end{proof}

The structure of the uniform measure of the tip-top action is then
entirely described by the following result. We re-obtain in particular
the values for the characteristic root and for the Parry cocycle found
in~\S~\ref{sec:cont-brute-force} for the $4$-Rabati tiling in line,
and also in~\S~\ref{sec:using-theta-polyn} by using the theta
polynomial for the characteristic root.

\begin{theorem}
\label{thr:7}
Let $\M$ be a trace monoid, of M\"obius polynomial~$\mu_\M$
{\normalfont(\S~\ref{sec:growth-series})}, and let $p_0\in(0,1)$ be
the root of smallest modulus of~$\mu_\M$\,.

 Then the characteristic
  root $t_0$ of the tip-top action $\C\times\M\to\C$  is given by
  $t_0=\sqrt{p_0}$\,. The Parry cocyle of the action is given by:
  \begin{gather}
\label{eq:47}
    \forall(\gamma,\gamma')\in\C\times\C\quad
    \Gamma(\gamma,\gamma')=t_0^{|\gamma'|-|\gamma|}\,.
  \end{gather}
\end{theorem}

\begin{proof}
  Let $\nu=(\nu_\alpha)_{\alpha\in\C}$ be the uniform measure of the
  tip-top action, and let $(f_\gamma)_{\gamma\in\C}$ be the associated
  M\"obius valuation.

  Consider a clique $\gamma\in\C$. Then $\gamma$ is also an action
  enabled at state~$\gamma$, the action of which yields to the empty
  clique~$1$.  Given the result of Corollary~\ref{cor:4},
  point~\ref{item:14}, on the one hand, and the form~(\ref{eq:41}) of
  the uniform measure, we have thus:
  \begin{align*}
    f_\gamma(\gamma)&=1=t_0^{|\gamma|}\Gamma(\gamma,1)\,,&
\Gamma(\gamma,1)&=t_0^{-|\gamma|}\,.
  \end{align*}

  The cocyle relation yields
  $\Gamma(1,\gamma)=\Gamma(\gamma,1)^{-1}=t_0^{|\gamma|}$\,. Using
  again the cocyle relation, we obtain for any two cliques
  $\gamma,\gamma'\in\C$:
\begin{gather*}
  \Gamma(\gamma,\gamma')=\Gamma(\gamma,1)\Gamma(1,\gamma')=t_0^{|\gamma'|-|\gamma|}\,,
\end{gather*}
proving~(\ref{eq:47}).

We obtain also the following form for the restriction to $\C$
of~$f_1(\cdot)$:
\begin{gather*}
\forall\gamma\in\C\quad  f_1(\gamma)=t_0^{|\gamma|}\Gamma(1,\gamma)=t_0^{2|\gamma|}\,.
\end{gather*}

Let $h_1(\cdot)$ be the M\"obius transform of~$f_1(\cdot)$. Then by
the M\"obius condition for~$f_1$\,, we have $h_1(1)=0$ which is
equivalent to $\mu_\M(t_0^2)=0$\,, hence $t_0^2$ is a root
of~$\mu_\M$\,. The M\"obius conditions also impose $h_1(\gamma)\geq0$
for all non empty cliques $\gamma\in\Cstar$. Lemma~\ref{lem:8} below
shows that among the roots of~$\mu_\M$\,, only the root of smallest
modulus satisfies these conditions. This completes the proof.
\end{proof}

In the course of the above proof, we have used the following
lemma. The key ingredient is the uniqueness of uniform measures shown
in~\cite{abbesmair14}.

\begin{lemma}
\label{lem:8}
Let $\M$ be a trace monoid, of M\"obius polynomial $\mu_\M(t)$\,. Let
$t$ be a non negative real, let $f:\M\to\bbR$ be the uniform valuation
defined by $f(x)=t^{|x|}$\,, and let $h:\C\to\bbR$ be the M\"obius
transform of~$f\rest\C$\,. Assume that $h(1)=0$ and that $h(c)\geq0$
for all non empty cliques $c\in\Cstar$. Then $t$ is the root of
smallest modulus of~$\mu_\M$\,.
\end{lemma}

\begin{proof}
  Let $D$ be the dependence relation of the monoid $\M=\M(\Sigma,I)$,
  defined by $D=(\Sigma\times\Sigma)\setminus I$. The monoid $\M$ is
  said to be irreducible---not to be confused with the irreducibility
  of actions---if $(\Sigma,D)$ is connected as an undirected graph.

  The proof of the lemma is based on the following fact: there exists
  a unique Bernoulli measure $\nu$ on $\BM$ such that $\nu(\up
  x)=t^{|x|}$ for all $x\in\M$ and for some real $t>0$, and $t$ is the
  root $p_0$ of smallest modulus of~$\mu_\M$\,. This is proved in
  Theorem~1.6 of~\cite{abbesmair14} if $\M$ is irreducible. But it is
  true also if $\M$ is not irreducible, by considering the projection
  of $\nu$ to the component of~$\M$, of which the smallest root of the
  M\"obius polynomial coincides with~$p_0$
  (see~\cite[Prop.~4.6]{krob03} and~\cite[Prop.~1]{abbes:_unifor}).

  Now, given the real number $t$ and the valuation $f$ as in the
  statement of the lemma, we claim that there exists a Bernoulli
  measure $\nu$ on $\BM$ such that $\nu(\up x)=f(x)$. Indeed,
  considering the unique total and trivial action of $\M$ on the
  singleton~$\{*\}$, then the valuation $f$ can be seen as a fibred
  M\"obius valuation with respect to this action.  An application of
  Theorem~\ref{thr:2} yields the existence of~$\nu$. Hence $t=p_0$\,.
\end{proof}

\section{Open questions}
\label{sec:open-questions}

One of the first questions left open is the uniqueness of the uniform measure for concurrent systems. It turns out to be a much more difficult question than in the case of simple trace monoids, where the Perron-Frobenius theorem allows almost directly to conclude to the uniqueness of the uniform measure. In general, the   graph of states-and-cliques of a concurrent system does not define an irreducible adjacency matrix, even if the concurrent system is irreducible. This prevents to use the strong results from Perron-Frobenius theory to conclude to the uniqueness of the uniform measure. Yet, it seems reasonable to conjecture that, for every concurrent system $(\M,\XS,\bot)$, there exists a unique measure of the form $\nu_\alpha(\up x)=t^{|x|}\Gamma(x,\alpha\cdot x)$ for $x\in\M_\alpha$ and $\nu_\alpha(\up x)=0$ for $x\notin\M_\alpha$\,, with $\Gamma$ a positive cocyle.

In relation with the above question is the nature of the Möbius matrix $\mu(t_0)$, where $t_0$ is the characteristic root of the concurrent system. Since $\det\bigl(\mu(t_0)\bigr)=0$, the kernel of $\mu(t_0)$ has positive dimension. Actually, for any fixed state~$\alpha_0$, it is easy to see that $\bigl(\Gamma(\alpha_0,\beta)\bigr)_{\beta\in \XS}$ is a non zero vector of $\ker\bigl(\mu(t_0)\bigr)$. And conversely, if $\ker\bigl(\mu(t_0)\bigr)$ has dimension greater than~$1$, one can construct a different cocycle yielding a different uniform measure. Hence the above conjecture implies that $\dim\bigl(\ker\bigl(\mu(t_0)\bigr)\bigr)=1$.

\bibliographystyle{plain}
\bibliography{async}

\end{document}